\documentclass[a4paper,11pt,leqno,english]{smfart}
\usepackage{aeguill}
\usepackage{enumerate}
\usepackage{amssymb,amsmath,latexsym,amsthm}
\usepackage[T1]{fontenc}
\usepackage{geometry}   
\usepackage{url}   
\usepackage[frenchb, main=english]{babel}
\usepackage[utf8]{inputenc}   
\usepackage{mathrsfs} 
\usepackage{xcolor}
\usepackage{comment}
\definecolor{violet}{rgb}{0.0,0.2,0.7}
\definecolor{rouge2}{rgb}{0.8,0.0,0.2}
\usepackage{tikz}
\usepackage{empheq}
\usepackage{tikz-cd}
\usetikzlibrary{matrix,arrows,decorations.pathmorphing}
\usepackage{hyperref}
\hypersetup{
    bookmarks=true,         
    unicode=false,          
    pdftoolbar=true,        
    pdfmenubar=true,        
    pdffitwindow=false,     
    pdfstartview={FitH},    
    pdftitle={},    
    pdfauthor={},     
    colorlinks=true,       
   linkcolor=rouge2,          
    citecolor=violet,        
    filecolor=black,      
    urlcolor=cyan}           
\usepackage{enumitem}
\usepackage{mathpazo}
\usepackage{newunicodechar}
\newunicodechar{°}{\ensuremath{^\circ}}

\newcommand{\R}{\mathbb{R}}

\newcommand{\Q}{\mathbb{Q}}

\newcommand{\N}{\mathbb{N}}

\renewcommand{\d}{\partial}

\newcommand{\vp}{\varphi}

\newcommand{\wt}{\widetilde}

\renewcommand{\O}{\mathcal{O}}

\newcommand{\ep}{\varepsilon}
\renewcommand{\epsilon}{\varepsilon}

\newcommand{\la}{\langle}

\newcommand{\ra}{\rangle}

\renewcommand{\ge}{\geqslant}
\renewcommand{\le}{\leqslant}
\renewcommand{\leq}{\leqslant}
\renewcommand{\geq}{\geqslant}
\newcommand{\Ric}{\mathrm{Ric} \,}
\newcommand{\codim}{\mathrm{codim}}

\newcommand{\om}{\omega}
\newcommand{\omi}{\omega_{\infty}}

\newcommand{\omke}{\omega_{\rm KE}}
\newcommand{\omkey}{\omega_{{\rm KE}, y}}

\newcommand{\omked}{\omega_{{\rm KE}, \delta}}
\newcommand{\omcan}{\omega_{\rm can}}
\newcommand{\ddc}{dd^c}

\newcommand{\Supp}{\mathrm {Supp}}

\newcommand{\vpd}{\vp_{\delta}}

\renewcommand{\tt}{\theta_t}

\newcommand{\ssm}{\smallsetminus}

\newcommand{\vpde}{\varphi_{\delta, \ep}}
\newcommand{\pde}{\phi_{\delta, \ep}}

\newcommand{\dbar}{\bar \partial}

\newcommand{\cF}{\mathcal F}

\usepackage{graphicx}

\numberwithin{equation}{section}

\setdescription{labelindent=\parindent, leftmargin=2\parindent}
\setitemize[1]{labelindent=\parindent, leftmargin=2\parindent}
\setenumerate[1]{labelindent=0cm, leftmargin=*, widest=iiii}

\theoremstyle{plain}
\newtheorem{theo}{Theorem}[section]
\newtheorem{prop}[theo]{Proposition}
\newtheorem{coro}[theo]{Corollary}

\newtheorem{lemm}[theo]{Lemma}
\newtheorem{conj}[theo]{Conjecture}
\newtheorem{set}[theo]{Setting}

\newtheorem{defin}[theo]{Definition}
\newtheorem{defiprop}[theo]{Definition-Proposition}

\newtheorem{rema}[theo]{Remark}

\newtheorem{claim}[theo]{Claim}
\setlist[enumerate]{label=(\thetheo.\arabic*), before={\setcounter{enumi}{\value{equation}}}, after={\setcounter{equation}{\value{enumi}}}}

\newtheorem{bigthm}{Theorem}

\renewcommand{\theequation}{\thesection  .\!\! \arabic{equation}}

\begin{document}

\title[Variation of singular Kähler-Einstein metrics]{Variation of singular Kähler-Einstein metrics:\\ positive Kodaira dimension}

\author{Junyan Cao}
\address{Institut de Mathématiques de Jussieu, Université Paris 6, 4 place Jussieu, 75252 Paris, France}
\email{junyan.cao@imj-prg.fr}

\author{Henri Guenancia}
\address{Institut de Mathématiques de Toulouse; UMR 5219, Université de Toulouse; CNRS, UPS, 118 route de Narbonne, F-31062 Toulouse Cedex 9, France}
\email{henri.guenancia@math.cnrs.fr}

\author{Mihai P\u{a}un}
\address{Institut für Mathematik, Universität Bayreuth, 95440 Bayreuth, Germany}
\email{mihai.paun@uni-bayreuth.de}

\begin{abstract} 
  Given a K\"ahler fiber space $p:X\to Y$ whose generic fiber is of general type, we prove that the fiberwise singular Kähler-Einstein metric induces
a semipositively curved metric on the relative canonical bundle $K_{X/Y}$ of $p$. We also propose a conjectural generalization of this result for relative
  twisted Kähler-Einstein metrics. Then we show that our conjecture holds true
  if the Lelong numbers of the twisting current are zero. Finally, we explain the relevance of our conjecture for the study of fiber-wise Song-Tian metrics (which represent the analogue of KE metrics for fiber spaces whose generic fiber has positive but not necessarily maximal Kodaira dimension).  
\end{abstract}
\date{\today}
\maketitle
\tableofcontents

\section*{Introduction}

\noindent Let $p:X\to Y$ be a K\"ahler fiber space. By this we mean that $p$ is a
proper, surjective holomorphic map with connected fibers such that the total space $X$ is  Kähler. Important questions in birational geometry (such as e.g. Iitaka
$C_{nm}$ conjecture) are treated by investigating the properties of direct images
\begin{equation}\label{mp1}
p_\star(mK_{X/Y})
\end{equation}
where $m$ is a positive integer, and $K_{X/Y}:=  K_X- p^\star(K_Y)$ is the relative canonical bundle of the map $p$. In other words, one considers the
variation of the pluricanonical linear series $\displaystyle
H^0(X_y, mK_{X/Y}|_{X_y})$ for $y\in Y$ and some fixed $m\gg 0$.
\medskip

\noindent In this article we will adopt a slightly different point of view by working
with an
object which ``encodes'' the asymptotic behavior of the entire canonical ring $\displaystyle
\bigoplus_m p_\star(mK_{X/Y})$. If the generic fiber of $p$ is of general type, then this turns out to be the singular K\"ahler-Einstein metric. 
The direct image \eqref{mp1} is positively curved, and our main concern in this article is to show that the same holds true for the metric induced on $K_{X/Y}$ by fiber-wise singular K\"ahler-Einstein metrics.

\subsection*{General set-up and main results}

Let $p:X\to Y$ be a K\"ahler fiber space. We will systematically use the notation
$Y°\subset Y$ for a set contained in the regular values of the map $p$ such that
the complement $Y\setminus Y°$ is analytic. Let $X° :=p^{-1} (Y°)$ be its inverse image.

Let $(L, h_L)$ be a $\Q$-line bundle endowed with a singular metric $h_L$
whose curvature current is positive, i.e.
\begin{equation}\label{mp03}
\Theta_{h_L}(L)\geq 0.
\end{equation} 
We assume for the moment that the relative adjoint bundle $K_{X/Y}+L$ is $p$-big. In many important geometric settings (including the case $\mathcal I(h_L)= \mathcal O_X$) for every $y\in Y$ general enough
there exists a unique closed positive current $\displaystyle \om_{{\rm KE},y}\in c_1(K_{X_y}+L)$ such that
\begin{equation}\label{mp3}
\Ric \om_{{\rm KE},y}= -\om_{{\rm KE},y}+\Theta_{h_L}(L)
\end{equation}
The precise framework for \eqref{mp3} to hold will become clear in Section~\ref{setting}.
In what follows $\om_{{\rm KE},y}$  will be referred to as \emph{singular K\"ahler-Einstein metric} by analogy with the case $L$ trivial and $K_{X_y}$ ample.  

\medskip

\noindent The results we establish in this article are converging towards the following general problem.

\begin{conj}
\label{conjecture}
In the above set-up, the relative Kähler-Einstein metrics $(\om_{{\rm KE},y})_{y\in Y°}$ induce a metric $e^{-\phi_{\rm KE}}$ on $K_{X°/Y°} +L|_{X°}$ which is positively curved and which extends canonically across $X\ssm X°$ to a positively curved metric on $K_{X/Y}+L$.
\end{conj}

As consequence of important approximation results in pluripotential theory
we show that a much more general form of the conjecture above would follow provided that
one is able to deal with the case where
$\Theta_{h_L}(L)$ equals the current of integration along a divisor with simple normal crossings support and coefficients in $(0,1)$ plus a smooth form, cf. Theorem~\ref{thmaaa}.
\medskip

\noindent Our main theorem states the following.
\begin{bigthm}
\label{thmaint} Conjecture \ref{conjecture} holds true if the Lelong numbers of the curvature current corresponding to $h_L$ are zero on the $p$-inverse image of a Zariski open subset of $Y$.
\end{bigthm}

For example, if $L=0$ then Theorem \ref{thmaint} shows that the metric on $K_{X/Y}$ induced by the fiber-wise KE current is positively curved. 
\medskip

One of the main motivations for Conjecture \ref{conjecture} will become clear from the context we next discuss. Let $p:X\to Y$ be a K\"ahler fiber space, and let $B$ be an effective $\Q$-divisor on $X$ with coefficients in $(0,1)$ such that $B|_{X_y}$ has simple normal crossings support for $y$ generic. Assume furthermore that $K_{X_y}+B|_{X_y}$ has \textit{positive} Kodaira dimension. Here we use the notation ``$B$'' rather than
$L$ in order to emphasize that the metric is fixed.

There exists a relative version of the so-called \emph{canonical metric} introduced by Song and Tian \cite{ST12} and generalized by Eyssidieux, Guedj and Zeriahi \cite{EGZ16}. It is defined on the base $Z'$ of a birational model $q':X'\rightarrow Z'$ of the relative Iitaka fibration $q:X\dashrightarrow Z$ over $Y$, cf. Section~\ref{ssec:intkod} for more details. In case of a family $p$
whose generic fiber has maximal Kodaira dimension, the metric in \cite{ST12} coincides with the singular K\"ahler-Einstein metric (up to a birational transformation).

\begin{bigthm}
\label{corb}
Let $p:X\to Y$ be a K\"ahler fiber space such that for $y$ generic, $\kappa (K_{X_y}+B_y) >0$. 
Let $f: X\dashrightarrow Z$ be the relative Iitaka fibration of $K_{X/Y}+B$, and let $f':X'\rightarrow Z'$ a birational model of $f$ such that $X'$ and $Z'$ are smooth.
\begin{center}
\begin{tikzcd}
X'\arrow{d} \arrow{rr}{f'} & & Z' \arrow{d}\\
X\arrow[dashed]{rr}{f}\arrow[swap]{rd}{p} & & Z \arrow{dl}\\
& Y &
\end{tikzcd}
\end{center}
Let $\omega_{{\rm can}, y}$ be the canonical metric on $Z'_y$ of the pair $(X'_y,B'_y)$; it induces a current $\omcan°$ over the smooth locus of $Z'\to Y$. 

\noindent
Moreover we assume that Conjecture \ref{conjecture} holds true. Then
the current $\omcan°$ is positive and extends canonically to a closed positive current on $Z'$.
\end{bigthm}



\subsection*{Previously known results} There are basically two types of techniques used in order to address the questions we are interested in
here, due to Schumacher and Tsuji in \cite{Schum08} and \cite{Tsuji10}, respectively. The former concerns the smooth case (e.g. $K_{X_y}$ ample) and it is based on a maximum principle. The later consists in 
showing that non-singular KE metrics can be obtained by an iteration scheme involving pluricanonical sections normalized in a specific way. Both methods have their advantages and flaws. For example, it is difficult to conceive that
Schumacher method can be used in the presence of base points. Also, at first sight the
method of Tsuji looks very general. However, it uses in an essential manner the
asymptotic expansion of Bergman kernels, which depends on at least two derivatives of the metric. This is the main reason why we cannot deal with
the general case of a line bundle $(L, h_L)$ as in Conjecture \ref{conjecture}.

In the paragraph that follows we recall the definitions of relative (singular) Kähler-Einstein metrics and collect some earlier results.

\subsubsection*{Relative singular Kähler-Einstein metrics}

A singular Kähler-Einstein metric is a generic term to refer to a non-smooth, closed, positive, $(1,1)$-current $\om$ that satisfies a Kähler-Einstein like equation in a weak sense. Among the most natural examples are: Kähler-Einstein with conic singularities, mentioned above, cf. also \cite{Brendle, CGP, JMR}, Kähler-Einstein metrics on singular varieties, cf. \cite{EGZ, BBEGZ, BG}. These metrics are obtained by solving an equation of the form 
\begin{equation}
\label{KE2}
\Ric \om = \lambda \om + T
\end{equation}
on a compact Kähler manifold $X$, where $T$ is a closed $(1,1)$-current (e.g. the current of integration along a $\R$-divisor with coefficients in
$]-\infty, 1]$). The cohomology class $\alpha\in H^{1,1}(X, \mathbb R)$ of $\om$ is determined by the equation unless $\lambda=0$, and it may be degenerate. That is, instead of being Kähler, $\alpha$ may be semipositive and big, or even merely big. The singularities of $\om$ may then appear because of the singularities of $T$ or the non-Kählerness of $\{\om\}$. The singularities of the first type are rather well known when $T$ is a current of integration along an effective divisor with snc support (one gets conic or cusp singularities, cf. e.g. \cite{KobR, Tia, G12}), but they are mostly mysterious in the second case with a few numbers of exceptions like when $X$ is a resolution of singularities of a variety $Y$ with orbifold singularities, or isolated conical singularities cf. \cite{HS}, and $\alpha$ is the pull-back of a Kähler class on $Y$.

\subsubsection*{Earlier results}

If the generic fiber has ample canonical bundle, that is, if $K_{X_y}$ is ample for any $y\in Y°$, then it follows from the Aubin-Yau theorem \cite{Aubin} \cite{Yau78} that one can endow each smooth fiber with a Kähler-Einstein metric with $\lambda=-1$. This induces a metric on $K_{X/Y}|_{X°}$
whose curvature form $\om_{\rm KE}° $ is smooth (by implicit function theorem). Moreover, the restriction $\om_{\rm KE}° |_{X_y}$ coincides with the KE metric. The surprizing important fact is that $\om_{\rm KE}° \geq 0$ on $X°$,
as 
it has been showed by Schumacher \cite{Schum08} and independently by Tsuji \cite{Tsuji10}.

\noindent $\bullet$   
Following Schumacher's strategy, one obtains in \cite{Paun12} a generalization of this result to the Kähler setting (including the extension property) only assuming that $K_{X_y}+\{\beta\}|_{X_y}$ is relatively ample for some smooth, semipositive, closed $(1,1)$-form $\beta$ on $X$.

Based on this approach again, the second name author studied the conic analogue of these questions, cf. \cite{Gue16}: let $B=\sum b_i B_i$ is a divisor with snc support  on $X$ and coefficients in $(0,1)$ and assume that $K_{X_y}+B|_{X_y}$ is ample for $y\in Y°$, the the relative conical Kähler-Einstein metric solution of $\Ric \om_y=-\om_y+[B|_{X_y}]$ induces a singular $(1,1)$ current $\om_{\rm KE}°$ on $X°$ that is positive, and extends canonical to a positive current $\omke \in c_1(K_{X/Y}+B)$.

We refer to \cite{Schum}, \cite{Choi}, \cite{BS}
for other applications of this method.
\smallskip

\noindent  $\bullet$ In the case of a manifold with ample canonical bundle,
Tsuji observes that $\om_{\rm KE}°$ is the limit of relative Bergman kernels whose variation is known to be semipositive cf. \cite{BP}.
The metric induced by fiber-wise Bergman kernels extends, cf. \emph{loc. cit}.
Therefore
$\om_{\rm KE}°$ extends canonically to a current $\omke \in c_1(K_{X/Y})$ on $X$.
See \cite{TsujiNag} for potential applications.
\medskip

\noindent   $\bullet$ In a more general singular case  $K_{X_y}+B_y$ big, the fiberwise Kähler-Einstein metrics pick up singularities that are yet to be understood, and neither of the previous approaches seem to work.\\

\subsection*{About the proof}
The strategy of the proof of Theorem~\ref{thmaint} is explained in detail at the beginning of Section~\ref{proofthm} and consists in realizing the singular Kähler-Einstein metric as a limit of suitably chosen and renormalized Bergman kernels, as those are known to vary in a psh way by \cite{BP}. Although the global scheme of our arguments is similar to \cite{Tsuji10}, the level of difficulties induced by the presence of base points in the problems we are treating here 
is far more severe.  
In order to overcome them, one has to resort to using numerous intricate approximation processes. Ultimately, our feeling is that the room to manoeuvre is so small that Theorem~\ref{thmaint} is probably close to the optimal result that our method can reach, aside from the orbifold case discussed in Section~\ref{orbifold}.

\subsection*{Acknowledgments} We would like to thank Sébastien Boucksom, Tristan Collins, Vincent Guedj, Christian Schnell, Song Sun, Valentino Tosatti and Botong Wang for numerous useful discussions about the topics of this paper. 
This work has been initiated while H.G. was visiting KIAS, and it was carried on during multiple visits to UIC as well as to IMJ-PRG; he is grateful for the excellent working conditions provided by these institutions. During the preparation of this project, the authors had the opportunity to visit FRIAS on several occasions and benefited from an excellent work environment. 

H.G. is partially supported by NSF Grant DMS-1510214, and M.P is partially supported by NSF Grant DMS-1707661 and Marie S. Curie FCFP.

\setcounter{tocdepth}{2}

\section{Pluricanonical sections and singular Kähler-Einstein metrics}

\label{setting}

Let $X$ be a compact Kähler manifold of dimension $n$. 
Let $(L,h_L)$ be a $\Q$-line bundle endowed with a possibly singular hermitian metric 
$h_L=e^{-\phi_L}$ with positive curvature, that is $$\Theta_{h_L}(L)=\ddc \phi_L \ge 0$$ in the sense of currents.

We now recall the definition of K\"ahler-Einstein metric for the pair $(X, L)$ in case $K_X +L$ is big. This definition has been given by \cite[\S~6]{BEGZ} when $L=0$, and can be easily adapted to our slightly more general context.

\begin{defiprop}
\label{KElog}
Let $X$ be a compact Kähler manifold, let $(L,h_L)$ be a $\Q$-line bundle endowed with a singular hermitian metric $h_L=e^{-\phi_L}$ with positive curvature, that is, $\Theta_{h_L}(L)\ge 0$ in the sense of currents. 
We assume moreover that 
\begin{enumerate}
\item \label{dbig}The $\Q$-line bundle $K_X +L$ is big. 
\item \label{fg} The algebra $R(X,L)=\bigoplus_{m\ge 0 } H^0(X,\lfloor m(K_{X} +L) \rfloor)$ is finitely generated.
\item \label{nef} For every $p\in\N$ and every $s\in H^0 (X, p (K_X +L))$, we have $\int_X |s|^{{2}/{p}} e^{-\phi_L} < +\infty$.  
\end{enumerate}
Then, there exists a unique closed, positive $(1,1)$-current $\omke$ on $X$ which satisfies the following conditions.
\begin{enumerate}
\item \label{min}The current $\omke$ belongs to the big cohomology class $c_1(K_X+L)$ and it has full mass, that is, $\int_{X}\la \omke^n \ra = \mathrm{vol}(K_X+L)$. 
\item \label{keeq}The current $\omke$ satisfies the following equation in the weak sense of currents
$$ \Ric \omke = -\omke +\frac{i}{2\pi}\Theta_{h_L} (L).$$
\end{enumerate}
\end{defiprop}

\begin{rema}
\label{rem:defi} {\rm Some remarks are in order. 

$(a)$ An important feature of this definition is that it is birationally invariant. More precisely, if $(X,L,e^{-\phi_L})$ satisfies conditions \ref{dbig}-\ref{nef} and if $\pi:X'\to X$ is any birational proper morphism, then so does $(X',L', e^{-\phi_{L'}})$ where $L':=\pi^*L, \phi_{L'}:=\pi^*\phi_{L}$. Furthermore, if $\omke'$ is the Kähler-Einstein metric of $(X',L', e^{-\phi_{L'}})$, then $\omke'= \pi^*\omke+[K_{X'/X}]$. \\
   
$(b)$  
Conditions \ref{fg} and \ref{nef} are automatically satisfied if the multiplier ideal sheaf of $h_L$ is trivial, that is, if $\mathcal I(h_L)=\mathcal O_X$. This is clear for \ref{nef}.  As for \ref{fg}, we use the following argument. As $K_X +L$ is big, $X$ is automatically projective and we have 
$$K_X +L\equiv_{\Q} A+E$$ 
for some ample $\Q$-line bundle $A$ and an effective $\Q$-divisor $E$. 
From the solution of the openness conjecture, cf. \cite{BoB15, GuanZhou15}, for some $m_0$ large enough, 
we have
\begin{equation}\label{opennes}
\mathcal I(e^{-\phi_L- \frac{1}{m_0} \phi_E})=\mathcal O_X .
\end{equation}
As the question is birationally invariant,  by Demailly's regularization theorem, after some birational morphism,
we can suppose that
$$L+\frac{1}{m_0} A \equiv_{\Q} B + H$$ 
for some effective $\Q$-divisor $B$ and a semi-ample $\Q$-line bundle $H$ such that $\phi_L$ is more singular than $\phi_B$, where $\phi_B$ a canonical singular weight attached to $B$, cf proof of Lemma~\ref{klt} for instance.  
Together with \eqref{opennes}, $B+\frac{1}{m_0} E$ is klt. 
Thanks to \cite{BCHM}, the canonical ring of $K_X+ (B+\frac{1}{m_0} E)+ H$ is finitely generated.  Combining this with the relation 
$$(1+\frac{1}{m_0} )(K_X+L) \sim_{\Q} K_X+ (B+\frac{1}{m_0} E)+ H ,$$  
the condition \ref{fg} is proved.\\

$(c)$ If $L$ corresponds to an effective, klt $\Q$-divisor $B$ and $\phi_L$ is the canonical singular weight on $B$, then one recovers the standard log Kähler-Einstein metric whose existence follows essentially from \cite{BEGZ}, cf. e.g. \cite[\S~2.3]{G2}. \\

$(d)$  The condition \ref{keeq} can be rewritten in terms of non-pluripolar Monge-Ampère equations as follows: 
$$\la (\ddc \phi_{\rm KE})^n \ra  = e^{\phi_{\rm KE}-\phi_L}$$
where $\phi_{\rm KE}$ is a local weight for $\omke$ and where $\la \cdotp^n \ra $ denotes the non-pluripolar Monge-Ampère operator, cf. \cite[Def~1.1 \& Prop~1.6]{BEGZ}. \\

$(e)$ We will see in the proof that $\omke$ has minimal singularities in the sense of \cite[Def~1.4]{DPS01}. Moreover,  if $h_L$ is smooth on a non-empty Zariski open subset of $X$, then one can prove that $\omke$ is smooth on a Zariski open set by reducing the problem to the semiample and big case and use \cite{EGZ}. 
To our knowledge, there is still no purely analytical proof of the generic smoothness as explained in the few lines following \cite[Thm.~C]{BEGZ}.

}
\end{rema}

\begin{proof}[Proof of Definition-Proposition~\ref{KElog}]

Set $R(X,L):=\bigoplus_{m\geq 0} H^0 (X, \lfloor m (K_X+L)\rfloor)$, and let us define $X_{lc}:=\mathrm{Proj} \, R(X,L)$ to be the log canonical model of $X$, cf. e.g. \cite[Def.~3.6.7]{BCHM}. Taking a desingularization of the graph of the natural birational map $f:X\dashrightarrow X_{lc}$, one get the following diagram
\begin{center}
\begin{tikzcd}
 & \arrow{ld}[swap]{\mu} \widetilde{X} \arrow{rd}{\nu} & \\
X\arrow[dashed]{rr}{f} & &X_{lc} 
\end{tikzcd}
\end{center}
and the following formula
\begin{equation}
\label{zard}
\mu^*(K_X+L)= \nu^*(K_{X_{lc}}+L_{lc})+F
\end{equation}
where $L_{lc}:=f_*L$ (recall that $f$ does not contract any divisor) and $F$ is an effective $\nu$-exceptional divisor. Clearly, $K_{X_{lc}}+L_{lc}$ is ample. Thus, setting $\wt L:=\mu^*L, A:=\nu^*(K_{X_{lc}}+L_{lc})$ and $E:=F+K_{\widetilde X/X}$, the decomposition 
\begin{equation*}
\label{zari}
K_{\widetilde X}+\wt L=A+E
\end{equation*}
is a Zariski decomposition of $K_{\widetilde X}+\wt L$, and we have $\int_{\widetilde X}|s|^{\frac 2p}e^{-\phi_{\wt L}}<+\infty$ for any section $s\in H^0(\widetilde X, p(K_{\widetilde X}+\wt L))$ and where $\phi_{\wt L}:=\mu^*\phi_L$. We want to solve the following equation
\begin{equation}
\label{eqxtilde}
  (\ddc \wt \phi)^n   = e^{\wt \phi + \phi_E-\phi_{\wt L}} 
\end{equation}
for $\wt \phi$ a bounded psh weight on $A$, where $\phi_E$ is the canonical singular weight attached to $E$. Thanks to the results in \cite[\S~2.3]{G2}, we are reduced to establishing the following property. 
\begin{equation}
\label{kol}
e^{\phi_E- \phi_{\wt L}}\in L^{1+\ep} \quad \mbox{for some} \,\,\, \ep>0.
\end{equation} 

Take $p$ large enough so that $p\wt L,pE$ are integral and $|pA|$ is basepoint free. 
Let $\{\tau_1 ,\cdots \tau_r\}$ be a basis of $H^0 (\wt X, pA)$. Then $\sum_{i=1}^r |\tau_i|^{2}$ is non-vanishing everywhere.
Let $s_{pE}$ be the canonical section of $pE$.
Thanks to \ref{nef}, we have
$$\int_{\wt X} (\sum_{i=1}^r |\tau_i|^{2})^{\frac{1}{p}}  | s_{pE}|^{\frac{2}{p}} e^{-\phi_{\wt L} } < +\infty .$$
Together with the fact that $\sum_{i=1}^r |\tau_i|^{2}$ is non vanishing everywhere, we get
\begin{equation}
\label{kol2}
e^{-\phi_{\wt L} + \phi_E } \in L^1_{\rm loc}
\end{equation} 

By applying the solution of the generalized openness conjecture \cite{GuanZhou15, BoB15} to the psh weight $\phi_{\wt L}+(1-\frac 1p) \phi_{pE}$,  we see that \eqref{kol} follows from \eqref{kol2}.

Now, define $\phi:=\wt \phi+\phi_F$. From the Zariski decomposition \eqref{zard}, it follows that the psh weight $\phi$ on $\mu^*(K_{X}+ L)$ has minimal singularities and it satisfies $\la (\ddc \phi)^n \ra =\la (\ddc \wt \phi)^n\ra = e^{\phi-\phi_{\wt L}}$ as the operator $\la \cdotp^n \ra $ puts no mass on proper analytic sets. There exists a unique psh weight $\phi_{\rm KE}$ on $K_X+L$ such that $\phi=\mu^*\phi_{KE}$. It has automatically minimal singularities and satisfies $\la (\ddc  \phi_{\rm KE})^n\ra = e^{\phi_{\rm KE}-\phi_{ L}}$; this ends the proof. 
\end{proof}

It will be convenient to use the following setting. 
\begin{set}
\label{set}
Let $p: X \rightarrow Y$ be a projective fibration between two smooth Kähler manifolds. Let $(L,h_L)$ be a holomorphic singular hermitian $\Q$-line bundle on $X$ such that $i\Theta_{h_L} (L) \geq 0$. Let $Y°\subset Y$ be a Zariski open subset such that where $p$ is smooth over $Y°$, and  set $X° :=p^{-1} (Y°)$. Assume that the additional two conditions are satisfied.
\begin{enumerate}
\item \label{pbig}The $\Q$-line bundle $K_{X} +L$ is $p$-big and for every $y\in Y°$, the algebra 
$$R(X_y ,L)=\bigoplus_{m\ge 0 } H^0(X_y ,\lfloor m(K_{X_y} +L_y) \rfloor)$$ 
is finitely generated.
\item\label{nefcon} Let $y\in Y°$.  For every $m\in\N$ and every $s\in H^0 (X_y , m (K_X +L))$, we have $\int_{X_y} |s|^{\frac{2}{m}} _{h_L} < +\infty$.  
\end{enumerate}
\end{set}
\medskip

\noindent We can now state the precise form of the conjecture already mentioned in
the introduction.
\begin{conj}
\label{conj}
In the Setting~\ref{set} above, the K\"ahler-Einstein metrics $(\om_{{\rm KE},y})_{y\in Y°}$ on the smooth fibers in the sense of Definition-Proposition \ref{KElog} induce a metric $e^{-\phi_{\rm KE}}$ on $K_{X/Y} +L$ over $X°$ such that: 
\begin{enumerate}
\item \label{main11} $i\Theta_{\phi_{\rm KE}} (K_{X/Y} +L) \geq 0$ on $X°$. 
\item \label{main12} The metric $e^{-\phi_{\rm KE}}$ extends canonically across $X\ssm X°$ and $i\Theta_{\phi_{\rm KE}} (K_{X/Y} +L) \geq 0$ on $X$.
\end{enumerate}
\end{conj}

The above conjecture is very general as it deals with a wide range of singular hermitian bundles $(L,h_L)$. A version of this is the following.

\begin{conj}
\label{conj2}
Let $p: X \rightarrow Y$ be a Kähler fiber space. Let $(L,h_L)$ be a holomorphic hermitian $\Q$-line bundle on $X$ such that 
\begin{enumerate}
\item $(L, h_{L}) =(B+\Lambda, h_{B}h_{\Lambda})$  where $B$ is an effective $\Q$-divisor and the restriction
on the generic fiber $B|_{X_y}$ is klt with simple normal crossings support, $h_B$ is the canonical singular metric on $B$ and $\Lambda$ is a $\Q$-line bundle with a smooth hermitian metric $h_{\Lambda}$ such that $\Theta_{h_{\Lambda}} (\Lambda) \geq 0$ on $X$.
\item The $\Q$-line bundle $K_{X}+L$ is $p$-big and admits a relative Zariski decomposition, i.e.  $K_X+L \equiv_\mathbb{Q} A+E$ for some relatively semiample and big $\mathbb{Q}$-line bundle $A$ and an effective $\mathbb{Q}$-divisor $E$ such that the natural map 
$$p^*p_*\mathcal O_X(mA) \longrightarrow p^*p_*\mathcal O_X(mA+mE)$$ 
is a sheaf isomorphism over $Y°$ for any $m$ divisible enough.
\end{enumerate}
Then, the relative Kähler-Einstein metric $e^{-\phi_{\rm KE}}$ induced on $K_{X/Y} +L$ over $X°$ satisfies $$\Theta_{\phi_{\rm KE}} (K_{X/Y} +L) \geq 0 \quad \mbox{on} \quad X°.$$ 
\end{conj}

\bigskip

\noindent
 Our first result, proved in Section~\ref{proof1}, is to reduce Conjecture~\ref{conj} to Conjecture~\ref{conj2}: 

\begin{theo}
\label{thmaaa}
Conjecture~\ref{conj2} implies Conjecture~\ref{conj}. 
\end{theo}

At this stage, we are not able to prove Conjecture~\ref{conj} (or Conjecture~\ref{conj2}) in full generality but only in the particular case where the metric $h_L=e^{-\phi_L}$ on $L$  has vanishing Lelong numbers; i.e. $ \forall x\in X, \,\nu(\phi_L,x)=0 $. The proof is given in Section~\ref{sec:conv} and relies on the approach developed in Section~\ref{sec:approx}, after a reduction step explained in Section~\ref{sec:reduction}. 

\begin{theo}
\label{thmaa}
Conjecture~\ref{conj} holds true provided that Lelong numbers of the metric $h_L=e^{-\phi_L}$ of $L$ are vanishing.  
\end{theo}

\section{Proof of Theorem~\ref{thmaaa}}
\label{proof1}

We have organized this section in the following way: first we show that \ref{main11} implies \ref{main12}. The proof of \ref{main11} will be given in the second part of our arguments.\\

\noindent
\textbf{The extension property.}

\noindent

\begin{prop}
\label{extension}
In the Setting~\ref{set}, the local weights of $\phi_{\rm KE}$ are locally bounded above near $X\ssm X°$.
\end{prop}

\begin{proof}
The proof of the Proposition follows very closely \cite[\S 3.3]{Paun12}, so we will mostly sketch the proof. 

Let $y\in Y°$ and let us pick any point $x$ in $X_y$. We choose a Stein neighborhood $\Omega$ of $x$ in $X$; we write $\Omega_y = \Omega \cap X_y$, choose
a potential $\tau_y$ of $\omkey$ such that the equation satisfied by $\tau_y$ on $\Omega_y$ is
$$\la (\ddc \tau_y)^n \ra = e^{\tau_{y}-\vp_L}\left| \frac {dz}{dt} \right|^2$$
where $\vp_L$ is a local weight for $h_L$ on $\Omega$, and the coordinates $(z_1, \ldots, z_n, t_1, \ldots, t_m)$ are chosen so that $p(\underline z, \underline t)=\underline t$.
We set $H_{m,y}:=\left\{f\in \mathcal O(\Omega_y); \int_{\Omega_y}|f|^2e^{-m\tau_y}\la(\ddc \tau_y)^n \ra \le 1\right\}$.
Note that $e^{-m\tau_y}\la (\ddc \tau_y)^n\ra = e^{-(m-1)\tau_y-u}\left| \frac {dz}{dt} \right|^2$ for some psh function $u$ on $\Omega$. Then, thanks to Demailly's approximation theorem, one has $$\tau(y)(x)= \lim_{m\to \infty}\sup_{f\in H_{m,y}} \frac 1 m \log |f(x)|$$
But for $f\in H_{m,y}$, Hölder's inequality yields 
\begin{equation}
\label{maj}
\int_{\Omega_y}|f|^{2/m}e^{-\tau_y}\la(\ddc \tau_y)^n\ra \le \left(\mathrm{vol}(K_{X_y}+L_y)\right)^{\frac{m}{m-1}}
\end{equation}
The right hand side is bounded above independently of $y$ and $m$; this can be seen for instance by finding a birational model $\pi:X'\to X$ where $\pi^*(K_{X}+L)$ has a relative Zariski decomposition $A+E$ so that the volume of $K_{X_y}+L_y$ is simply the intersection number $(A_y^n)$ which is independent of $y\in Y°$. Furthermore, the $L^{2/m}$ version of Ohsawa-Takegoshi extension theorem \cite{BP2} yields a holomorphic function $F$ on $\Omega$ that extends $f$ and such that
$$|F(x)|^{2/m} \le C_{\Omega} \int_{\Omega} |F|^{2/m} |dz|^2 \le C \int_{\Omega_y}|f|^{2/m}\left| \frac {dz}{dt} \right|^2\le C' \int_{\Omega_y}|f|^{2/m} e^{-\tau_y}\la (\ddc \vp_y)^n\ra$$
as $\vp_L$ is bounded above on $\Omega$. Moreover, the integral on the right hand side is bounded above uniformly in $y$ and $m$ by  \eqref{maj}. Therefore $\sup_{X_y}\tau_y \le C$ for a constant $C$ that uniform as long as $y\in Y°$ varies in compact subsets of $Y$. 
\end{proof}


\noindent
\textbf{Regularization}
\smallskip


\noindent 
Thanks to Proposition~\ref{extension}, Conjecture~\ref{conj} reduces to its first Item~\ref{main11}. That property is local on the base so from now on, the base $Y$ will be a small Stein open set. The rest of this section is devoted to showing that Conjecture~\ref{conj2} implies that Item~\ref{main11} in Conjecture~\ref{conj} holds. 

\begin{lemm}
\label{zar}
It is enough to prove Item~\ref{main11} of Conjecture~\ref{conj} when $K_{X}+L$ admits a relative Zariski decomposition, 
namely $K_X+L \equiv_\mathbb{Q} A+E$ for some relatively semiample and big $\mathbb{Q}$-line bundle $A$ and an effective $\mathbb{Q}$-divisor $E$ such that the natural map 
$$p^*p_*\mathcal O_X(mA) \longrightarrow p^*p_*\mathcal O_X(mA+mE)$$ 
is a sheaf isomorphism over $Y°$ for any $m$ divisible enough. 
\end{lemm}

\begin{proof}
By assumption, the $\mathcal O_Y$-algebra $\mathcal E:=\bigoplus_{m\ge 0} p_*(m (K_{X/Y}+L))$ is finitely generated. 
By blowing up the base locus of $\mathcal E$, we can find a birational map 
$\mu: \widetilde{X} \rightarrow X$ such that on the generic fiber $\widetilde{X}_y$ of $\mu\circ p$,
we have the Zariski decomposition of $K_{\widetilde{X}_y} + \mu^* L |_{\widetilde{X}_y}$.
Therefore, there exists a Zariski dense open subset $Y_0 \subset Y$ such that $K_{\widetilde{X}} + \mu^* L$ 
admits a relative Zariski decomposition $K_{\widetilde{X}} + \mu^* L = A+E$ on $(\mu\circ p)^{-1} (Y_0)$ and
for any $m$ divisible enough, the natural map
$$(\mu\circ p)^*(\mu\circ p)_*\mathcal O_{\widetilde{X}}(mA) \longrightarrow (\mu\circ p)^* (\mu\circ p)_*\mathcal O_{\widetilde{X}}(mA+mE) \qquad\text{on } (\mu\circ p)^{-1} (Y_0)$$
is an isomorphism.

\medskip

Now, let $\om_{{\rm KE}, \widetilde X}$ (resp. $\om_{{\rm KE}, X}$) be the relative Kähler-Einstein metric with respect to $(\widetilde X, \mu^*L, \mu^*\phi_L)$ (resp. $(X,L, \phi_L)$)
over $Y_0$. 
Thanks to Remark \ref{rem:defi} $(a)$, we have
$$\mu_* \om_{{\rm KE}, \widetilde X}=\om_{{\rm KE}, X}  \qquad\text{on } p^{-1} (Y_0) .$$
If we can prove that $\om_{{\rm KE}, \widetilde X} \geq 0$ on $(\mu\circ p)^{-1} (Y_0)$, then
$$\om_{{\rm KE}, X}  =\mu_* \om_{{\rm KE}, \widetilde X}  \geq 0 \qquad\text{on } p^{-1} (Y_0) .$$
Together with Proposition \ref{extension}, the lemma is proved.
\end{proof}

\begin{prop}
\label{klt}
It is enough to prove Item~\ref{main11} of Conjecture~\ref{conj} when $(L, h_{L}) =(B+\Lambda, h_{B}h_{\Lambda})$ 
where $B$ is an effective $\Q$-divisor and the restriction
on the generic fiber $B|_{X_y}$ is klt with normal crossing support, $h_B$ is the canonical singular metric on $B$ and $\Lambda$ is a $\Q$-line bundle with a smooth hermitian metric $h_{\Lambda}$ such that $i\Theta_{h_{\Lambda}} (\Lambda) \geq 0$ on $X$.
\end{prop}

\begin{proof}
We proceed in two steps. \\
 
\noindent
{\bf Step 1.} \textit{Reduction to the case where $\ddc \phi_L$ is a Kähler current}

\noindent
By Lemma~\ref{zar} above, one can assume that $K_X+L=A+E$ is a relative Zariski decomposition of $K_X+L$ on $X$. 
As $Y$ is Stein and $K_X +L$ is $p$-big,  $K_X +L$ is big on $X$. 
Therefore, there exists a weight $\phi_{0}$ with analytic singularities on $(K_X+L)$ such that $\ddc \phi_{0} >0$ on $X$.
Let us fix some small $\delta>0$. We set $L_{\delta}:=L+\delta(K_X+L)$ so that the relative Zariski decomposition of $K_X+L_{\delta}$ is $(1+\delta)A+(1+\delta)E$.
Let $\phi_L +\delta \phi_{0}$ be the weight on $L_{\delta}$. Then $\ddc (\phi_L+\delta \phi_{0})$ is a Kähler current.
To finish the proof of Step 1, it remains to prove that 
\begin{enumerate}
\item[$(i)$] The triplet $(X,L_{\delta}, e^{-\phi_L-\delta \phi_{0}})$ admits a relative Kähler-Einstein metric $\omked$ for $\delta > 0$ small enough.
\item[$(ii)$] We have $\omked \to \omke$ in the weak topology when $\delta$ approaches zero. 
\end{enumerate}

\medskip

To make notation more tractable, we will \---from now on and in this first step only \--- work on a fixed fiber $X_y$ and drop all indices $y$. \\

\noindent
$\bullet$ \textit{Proof of $(i)$}. 

\noindent
We know from \eqref{kol} that there exists $p>1$ such that  $e^{\phi_E-\phi_L}\in L^p_{}$.
Then $e^{(1+\delta)\phi_E-(\phi_L+\delta \phi_{0})}\in L^r$ for some $1<r<p$ as long as $\delta$ is small enough. Even better, 
\begin{equation}
\label{Lrd}
||e^{(1+\delta)\phi_E-(\phi_L+\delta \phi_{0})}||_{L^{r}(X)}\le C
\end{equation}
for some uniform $C>0$. Thanks to Definition-Proposition~\ref{KElog}, we get $(i)$.\\

\noindent
$\bullet$ \textit{Proof of $(ii)$}. 

\noindent
It requires more work. Let $\om_A$ be a smooth semipositive form in $c_1(A)$, let $h_L$ (resp. $h_E$) be a smooth hermitian metric on $\mathcal O_X(L)$ (resp. on $\mathcal O_X(E)$) and let $\omega$ be a reference Kähler form such that 
$$i\Theta_{\om}(K_X)+i\Theta_{h_L}(L)=\om_A+i\Theta_{h_E}(E).$$
Finally, let us choose potentials $\vp_L,\vp_0, \vp_E$ such that $i\Theta_{h_L}(L)+\ddc \vp_L=\ddc \phi_L$, $i\Theta_{\om}(K_X)+i\Theta_{h_L}(L)+\ddc \vp_0=\ddc \phi_0$ and $ i\Theta_{h_E}(E)+\ddc \vp_E=[E]$. The Kähler-Einstein metric $\omked$ can be written as $\omked=\om_{\delta}+(1+\delta)[E]$ where $\om_{\delta}=(1+\delta)\om_A+\ddc \vpd \in c_1((1+\delta)A)$ is a positive current with bounded potentials such that
 \begin{equation}
 \label{eqMA}
  ((1+\delta) \om_A+\ddc \vpd)^n  = e^{\vpd+(1+\delta)\vp_E-\vp_L-\delta \vp_0}\om^n
  \end{equation}
Let us write $d\mu:=e^{\vp_E-\vp_L}\om^n$ and $d\mu_{\delta}:= e^{(1+\delta)\vp_E-\vp_L-\delta \vp_0}\om^n$.

\begin{claim}
\label{claim1}
There exists a constant $C>0$ independent of $\delta$ such that 
\begin{equation}
\label{C0d}
||\vpd||_{L^{\infty}(X)}\le C.
\end{equation}
\end{claim}

\begin{proof}[Proof of Claim~\ref{claim1}]
A first trivial observation is that one can rewrite \eqref{eqMA} as a Monge-Ampère equation in a fixed cohomology class as follows
 $$(\om_A+\ddc \frac{1}{1+\delta}\vpd)^n  = e^{\vpd+(1+\delta)\vp_E-\vp_L-\delta \vp_0-n\log(1+\delta)}\om^n$$

Thanks to the a priori estimates established in \cite{EGZ}, the claim comes down to showing that there exists a uniform $C>0$ such that 
\begin{equation}
\label{bornesup}
\sup_{X}\vpd \le C
\end{equation}
and that $e^{(1+\delta)\vp_E-\vp_L-\delta \vp_0}$ admits uniform $L^p$ bounds for some $p>1$; but we already know that from \eqref{Lrd}. Let us prove \eqref{bornesup} now. As $\vpd$ has bounded potentials, its Bedford-Taylor Monge-Ampère has full mass, i.e.  
$$\int_X e^{\vpd+(1+\delta)\vp_E-\vp_L-\delta \vp_0}\om^n = (1+\delta)^n (A^n)$$ 
and, in particular, the integral $\int_X e^{\vpd+(1+\delta_0) \vp_E}dV$ is uniformly bounded above for $\delta_0>0$ fixed. An application of Jensen's inequality yields $\int_X \vpd dV \le C$, and the bound \eqref{bornesup} then follows from standard properties of quasi-psh functions. 
\end{proof}

The proof of Item $(ii)$ above now follows from the following claim. 

\begin{claim}
\label{claim2}
When $\delta$ approaches zero, the function $\vpd$ converges weakly to $\vp$.  
\end{claim}

\begin{proof}[Proof of Claim~\ref{claim2}]
An equivalent formulation of the claim is that  
\begin{equation}
\label{convclaim}
\vpd-\sup_{X}\vpd \underset{\delta \to 0}{\longrightarrow} \wt \vp:=\vp-\sup_X \vp.
\end{equation}
 This is consequence of \cite[Thm.~4.5]{BG}, but the bound \eqref{C0d} actually makes the arguments much easier. We will only recall the main lines. First, one chooses a sequence $\delta_j$ such that $\vp_j:=\vp_{\delta_j}-\sup_X \vp_{\delta_j}$ converges weakly to some sup-normalized $\om_A$-psh function $\psi$; we want to show that $\psi=\wt \vp$.  We use the variational characterization
of $\vp_j$ as the supremum of the functional $\mathcal G_j=\mathcal E_j+\mathcal L_j$ acting on sup-normalized $(1+\delta_j)\om_A$-psh functions. Here, $\mathcal E_j$ is the usual energy functional attached to $(1+\delta_j)\om_A$ and $\mathcal L_j(\bullet)=-\log \int_Xe^{\bullet} d\mu_{\delta_j}$. Thanks to \eqref{Lrd} and \eqref{C0d}, the dominated convergence theorem implies
\begin{equation}
\label{ld}
\lim_{j\to +\infty} \mathcal L_j(\vp_j)= \mathcal L(\psi)
\end{equation}
Moreover, \cite[Lem.~4.6]{BG} implies that 
\begin{equation}
\label{eqd}
\varlimsup_{j\to +\infty} \mathcal E_j(\vp_j)\le  \mathcal E(\psi)
\end{equation}
As $\vp\in \mathrm{PSH}(X,(1+\delta_j)\om_A)$, one has automatically $\mathcal G_j(\vp_j) \ge \mathcal G(\wt \vp)$. Finally, as Bedford-Taylor product is continuous with respect to smooth convergence, one has $\lim_j \mathcal E_{j}(\wt \vp)=\mathcal E(\wt \vp)$. Putting these last two results together with \eqref{ld} and \eqref{eqd}, one finds
$$ \mathcal G(\psi) \ge \varlimsup_{j\to +\infty} \mathcal G_j(\vp_j) \ge \varlimsup _{j\to +\infty} \mathcal G_j(\wt \vp) =\mathcal G(\wt \vp)$$
hence the result. 
\end{proof}



In conclusion,  $\vpd$ converges weakly to $\vp$, hence $\omked$ converges to $\omke$. This argument was done fiberwise, but it clear that the weak convergence on the fiber implies the weak convergence in any small neighborhood of the given fiber as well. This proves $(ii)$ and completes Step 1. \\

\noindent
{\bf Step 2.} \textit{Reduction to the case where $\phi_L$ has analytic singularities}

\noindent
By Step 1, one can assume that $\ddc \phi_L$ is a Kähler current. By Demailly regularization theorem \cite{D2}, $\phi_L$ is the weak, decreasing limit of strictly psh weights $\phi_{L,\ep}$ on $L$ with analytic singularities, say with singularities along the analytic set $Z_{\ep}$. Taking a log resolution $\pi_{\ep}:X_{\ep}\to X$ of $(X,Z_{\ep})$, one can assume that $\pi^{*}\phi_{L,\ep}= \phi_{B_{\ep}}+ \phi_{A_{\ep}}$ where $\phi_{B_{\ep}}$ is the canonical singular psh weight on an effective normal crossing $\Q$-divisor $B_{\ep}$, and $\phi_{A,\ep}$ is a smooth psh weight on some $\Q$-line bundle $A_{\ep}$ with $\Theta_{\phi_{A,\ep}} ({A_\ep}) \geq 0$ on $X_\ep$.

After passing to another birational model if necessary, one can assume that over a generic fiber, we have a Zariski decomposition
$$K_{(X_{\ep})_y} +\pi_{\ep} ^* L_\ep |_{(X_{\ep})_y} = M_y +E_y ,$$
and $B_\ep |_{(X_{\ep})_y}  +E_y$ is normal crossing.
Let $\Gamma_{\ep}:=B_\ep   \wedge E$ be the common part of $B$ and $E$.
We have the following Zariski decomposition
$$K_{(X_{\ep})_y} +\big(B_\ep |_{(X_{\ep})_y}  -\Gamma_{\ep} |_{(X_{\ep})_y} \big)  + A_\ep |_{(X_{\ep})_y} = M_y +\big(E_y -\Gamma_{\ep} |_{(X_{\ep})_y}\big) .$$
Furthermore, thanks to \ref{nefcon} and the decreasing property of $(\phi_{L,\ep})$, we know that the divisor
$\big(B_\ep |_{(X_{\ep})_y}  -\Gamma_{\ep} |_{(X_{\ep})_y} \big)$ is klt on $(X_{\ep})_y$. 

Let $\om_{\ep}$ be the relative Kähler-Einstein metric of 
$(X_{\ep} \rightarrow Y, \pi_{\ep}^*L, \pi_{\ep}^* (\phi_{L,\ep}))$ and let $\om_{\ep}'$ 
be the relative Kähler-Einstein metric of $(X_{\ep} \rightarrow Y,  (B_{\ep}- \Gamma_{\ep}) + A_{\ep}, \phi_{B_{\ep}}-\phi_{\Gamma_{\ep}}+\phi_{A_{\ep}})$. By definition, we have
\begin{equation}\label{birarelaa}
\om_{\ep}=\om_{\ep}'+[\Gamma_{\ep}] .
\end{equation}
If Conjecture~\ref{conj} holds for $(X_{\ep} \rightarrow Y, (B_\ep -\Gamma_{\ep} ) + A_\ep, \phi_{B_{\ep}}-\phi_{\Gamma_{\ep}}+\phi_{A_{\ep}})$, thanks to \eqref{birarelaa} and Remark \ref{rem:defi}, it holds also for 
$(X \rightarrow Y, L, \phi_{L,\ep})$. Finally, when $\ep$ converges to $0$, the relative K\"ahler-Einstein metric of $(X\rightarrow Y, L, \phi_{L,\ep})$ converges to the relative Kähler-Einstein metric of $(X\rightarrow Y, L,\phi_L)$ as a direct consequence of the comparison principle. Therefore Theorem~\ref{thmaaa} is proved.
\end{proof}




\section{Proof of Theorem \ref{thmaa}}
\label{proofthm}
\noindent We will present our arguments in several steps, according to the following plan.
\begin{enumerate}

\item[(a)] \emph{It is enough to prove Theorem \ref{thmaa} in case $h_L$ non-singular.} This is based on two results: we first use that $h_L$ is limit of non-singular metrics whose negative part of the curvature tends to zero. Another important fact we are using is that the algebra associated to $K_X+ L+ \delta H$ is finitely generated, for any $H$ ample and
for any positive rational $\delta$.

\item[(b)] \emph{It is enough to prove Theorem \ref{thmaa} provided that $A$ is an ample $\Q$-bundle.} Remark that in general, the semi-ample part $A$ of the Zariski decomposition is not ample. In this second step we write
  $A$ as limit of ample bundles, and show
  that the solution of the resulting Monge-Amp\`ere equation converges to the singular KE metric.

\item[(c)] \emph{Reduction to the case $c_1(L)\in {\mathbb Z}$.} Let $p$ be a positive integer such that $pL$ is a line bundle. Then we write
  $p(K_X+ L)=  K_X+ (p-1)(K_X+ L)+ L$ and then we replace our initial $\Q$-bundle $L$ with the line bundle $L_p:= (p-1)(K_X+ L)+ L$. The problem is that we also have to replace $h_L$ with a positively curved metric on $L_p$. The metric on $L$ is given. It is less clear what should be the metric on $K_X+ L$, since it has to fulfill two conditions: its curvature must be positive, and in the relative setting (i.e. when we replace $X$ with a fiber of $p$) it must induce a positively curved metric on the twisted relative canonical bundle.
It seems impossible to achieve this in one single step. 
  What \emph{is} possible is to set up an iteration scheme so that the resulting limit coincides with the singular KE metric.

\item[(d)] \emph{If $L (= L_p)$ is a line bundle, show that the singular K\"ahler-Einstein metric corresponding to $(X, L)$ can be obtained as limit of iterated Bergman kernels.} We conclude by this fact, since the fiber-wise Bergman kernel metric
has the required curvature properties specified in (c) above.
  
\end{enumerate}
Also, at each step we establish the relevant convergence
results needed to conclude at the end.

\subsection{Reduction to the case $h_L$ non-singular}

\label{sec:reduction}
Let $(L,h_L=e^{-\phi_L})$ be a hermitian line bundle on a projective variety $X$ such that $K_X+L$ is big and $\phi_L$ is a psh weight with vanishing Lelong numbers.
Let $\phi$ be the weight on $K_X+L$ such that $\om_{\phi}:=\ddc \phi$ is the Kähler-Einstein metric of $(X,L,e^{-\phi_L})$, ie 
$$\Ric(\om_{\phi})=-\om_{\phi}+\ddc\phi_L.$$
Let $H$ be an ample line bundle on $X$, and let $\phi_H$ be a weight on $H$ such that $\ddc \phi_H$ is a Kähler form. Thanks to Demailly regularization theorem, there exists a family of smooth weights $\phi_{L,\ep}$ on $L$ such that
$$\phi_{L,\ep} \downarrow \phi_L \quad \mbox{and} \quad \ddc (\phi_{L,\ep}  +\ep\phi_H)\ge 0.$$
Now, let $\delta\ge \ep$ be a positive number, and let $\phi_{\delta,\ep}$ be the suitably normalized weight on $K_X+L+\delta H$ such that $\om_{\pde}:=\ddc \pde$ is the Kähler-Einstein metric of $(X,L+\delta H,e^{-\phi_{L,\ep}-\delta \phi_H})$, ie 
$$\Ric(\om_{\pde})=-\om_{\pde}+\ddc (\phi_{L,\ep}+\delta \phi_H).$$

\begin{prop}
\label{conver}
With the notation above, there exists a family of positive numbers $(\delta_{\ep})_{\ep>0}$ decreasing to zero such that $\om_{\phi_{\delta_\ep,\ep}}$ converges weakly to $\om_{\phi}$ when $\ep$ approaches zero. 
\end{prop}

As an consequence, one gets the following
\begin{coro}
\label{reduction}
It is sufficient to prove Theorem~\ref{thmaa} when $h_L$ is smooth.
\end{coro}

\begin{proof}[Proof of Proposition~\ref{conver}]
Let us start by setting some additional notation. Let $\theta$ (resp $\theta_L$) be a closed smooth $(1,1)$-form in the cohomology class $c_1(K_X+L)$ (resp. $c_1(L)$). Let $\om_H:=\ddc \phi_H$ and let $dV$ be a smooth volume form such that $-\Ric(dV)+\theta_L=\theta$. Finally, let $\vp_L,\vp_{L,\ep}$ be some quasi-psh functions such that $\theta_L+\ddc \vp_L=\ddc \phi_L$ (resp. $\theta_L+\ddc \vp_{L,\ep}=\ddc \phi_{L,\ep}$) and satisfying additionnally that $\vp_{L,\ep}\downarrow \vp_L$ when $\ep \downarrow 0$. Let $\vpde$ be the unique $(\theta+\delta \om_H)$-psh function with minimal singularities solution of 
$$(\theta+\delta \om_H+\ddc \vpde)^n=e^{\vpde-\vp_{L,\ep}}dV$$
whose existence is guaranteed by \cite{BEGZ} (cf. also \cite[Thm.~2.2]{G2}). When $\ep=0$, one writes $\vpd:=\vp_{\delta, 0}$, and one sets $\vp:=\vp_0$. Note that $\theta+\ddc \vp$ is the Kähler-Einstein metric of $(X,L,e^{-\phi_L})$.

For the time being, let $\delta>0$ be fixed. As $\vp_{L,\ep}$ decreases toward $\vp_L$ and $\vp_L$ has vanishing Lelong number, the convergence $e^{-\vp_{L,\ep}}\uparrow e^{-\vp} $ happens in any $L^p$ space for $p>0$. In particular, if follows from \cite[Thm.~4.2]{GLZ} that $\vpde$ converges weakly to $\vpd=\vp_{\delta, 0}$ when $\ep$ approaches zero. 

As $\om_H\ge 0$, the $\theta$-psh function $\vp$ is also $\theta+\delta \om_H$-psh and it is a subsolution of the equation
$$(\theta+\delta \om_H+\ddc \psi)^n=e^{\psi-\vp_{L}}dV$$
so one gets 
\begin{equation}
\label{min}
\vp\le \vpd
\end{equation}
for any $\delta>0$. Moreover, the same argument shows that $\vpd$ decreases when $\delta \downarrow 0$. Let $\vp^*:=\lim_{\delta \to 0} \vpd$. If we can prove that $\vp^*=\vp$, then we will be done. 

From \eqref{min}, one can deduce two things. First, $\vp^*$ is a $\theta$-psh function with minimal singularities. Also, the sequence $(\vpd)_{\delta>0}$ is locally bounded on the ample locus $\Omega$ of $K_X+L$. Because the Monge-Ampère operator is continuous with respect to bounded decreasing sequences, one finds that 
$$(\theta+\ddc \vp^*)^n=e^{\vp^*-\vp_{L}}dV \quad \mbox{ on } \Omega.$$
As the non-pluripolar Monge-Ampère operator does not put any mass to analytic sets, it follows that the previous equation is satisfied on the whole $X$. As $\vp$ and $\vp^*$ have minimal singularities, the currents $\theta+\ddc \vp$  and $\theta+\ddc \vp^*$ have full mass (almost by definition, cf remarks below \cite[Def.~2.1]{BEGZ}) and therefore
\begin{align*}
\int_X e^{\vp^*-\vp_{L}}dV &= \int_X (\theta+\ddc \vp^*)^n \\
&=\int_X (\theta+\ddc \vp)^n\\
&= \int_X e^{\vp-\vp_{L}}dV
\end{align*}
As $\vp \le \vp^*$ by \eqref{min}, we see that $\vp=\vp^*$ almost everywhere. As both functions are $\theta$-psh, they must agree on $X$. 
\end{proof}



\subsection{The approximation of $A$}
\label{sec:approx}
A first remark is that thanks to Lemma~\ref{zar} and Corollary~\ref{reduction}, one can assume that $h_L$ is smooth and that $K_X+L$ admits a relative Zariski decomposition over $X°$ (which denotes here the inverse image of a well-chosen open subset of $Y$),
$$K_X+L=A+E$$  

By Kodaira lemma, there exists an effective $\Q$-divisor $E_X$ such that
$A- E_X$ is $p$-ample. As $Y$ is chosen to be Stein, one can assume that
$A-E_X$ is globally ample. Therefore for each positive, small enough $\delta$ we
have
\begin{eqnarray}
K_X+L&=&A+E \label{decomposition}\\
&=&A_{\delta}+E_{\delta} \nonumber
\end{eqnarray}
where $A_{\delta} := (1-\delta) A+\delta (A-E_X)$ is ample and $E_{\delta}  := E +\delta E_X$ for any $\delta>0$.

\medskip

\noindent {\bf Convention.} \emph{For the rest of this subsection our results will exclusively concern the fibers $X_y$ of $p$. 
Since $y\in Y°$ is fixed, we will denote $X_y$ by $X$ and drop the index $y$ in the relevant line bundles and weights that will be considered here.}
\medskip

\subsubsection{Notations.}\label{not}

Let $\om_A$ in $c_1 (A)$ be a smooth,
semi-positive representative. We denote by $E:= \sum_{i=1}^k a_i E_i$ and let $E_X:=
\sum_{i=1}^k c_i E_i$ the divisors above where some of $a_i, c_i$ could be zero. 
 Since $\{A-E_X\}$ is a K\"ahler class we can fix a Kähler form
$\omega_0 \in \{A- E_X\}$. For each positive $\delta$ 
we obtain a Kähler form $$\om_{\delta}:= (1-\delta) \om_A+\delta \om_0 \in c_1(A_{\delta})$$
\noindent where $A_\delta:= A- \delta E_X$. We write
$$ E_{\delta}=\sum_{i=1}^k a_i^{\delta} E_i,$$
where 
\begin{equation}
\label{not2}
a_i^{\delta}=a_i+\delta c_i
\end{equation}  

Let $s_i$
be a defining section for $E_i$ and let
$h_{E_i}=e^{-\rho_i}$ be a non-singular Hermitian metric on $\mathcal O_X(E_i)$.
We obtain the metrics
$h_E =\prod h_{E_i}^{a_{i}}$ and $h_{E_{\delta}}=\prod h_{E_i}^{a_i^{\delta}}$ on $E$ and $E_{\delta}$ respectively.

We define 
$$|s_{E_\delta}|^{2\{\ell p\}}:=\prod_i |s_{E_i}|^{2 (\lceil \ell p a_i^{\delta}\rceil-
  \ell p a_i^{\delta})}$$
where $|s_{E_i}|^2$ denotes the squared norm of $s_{E_i}$ with respect to $h_i$.

\medskip

\noindent
The Kähler-Einstein metric $\om_{\vp}=\om_A+\ddc \vp$ of $(X,L, \phi_{L})$ satisfies the following Monge-Ampère equation on $X$:
\begin{equation}
\label{MAf}
(\om_A+\ddc \vp)^n = |s_E|^2e^{\vp- f_L} \om^n
\end{equation}
where $\om$ is a reference K\"ahler metric on $X$ and $f_L$ is the unique smooth function on $X$ such that
\begin{equation}
\label{mp33}
\om_A+ \Theta(E)+ dd^cf_L= \Theta_{h_L}(L)+ \Theta_{\om}(K_X),  \qquad \int_XfdV_{\omega}= 0.
\end{equation}

\medskip

\noindent It will be convenient for later to fix some notations for the local expression of the objects above. Let $U\subset X$ be an open coordinate
subset such that the $\Q$-bundles above are trivial when restricted to $U$.

Let $\rho_i$ be the local weight of the metric $h_i$ with respect to a trivialization of $\mathcal O_X(E_i)|_U$. We will use the notation
$$\rho_{E_\delta}:= \sum a^\delta_i\rho_i$$
for the weight of the induced metric on $E_\delta$.


In a similar manner we introduce $\phi_A, \phi_0, \phi_L$ on $A, A- E_X$ and $L$ respectively, such that $\ddc \phi_A=\om_A$ and $\ddc \phi_0=\om_0$. Finally we consider for $\delta >0$
\begin{equation}
\label{phiad}
\phi_{A_{\delta}}:=(1-\delta) \phi_A+\delta \phi_0.
\end{equation}

We assume that the metrics $h_i$ are chosen such that $\phi_0=\phi_A-\delta \sum_i c_i \rho_{i}$.
\medskip

\noindent Expressed in terms of local weights and coordinates, the equality \eqref{mp33} becomes 
\begin{equation}
\label{mp63}
\phi_A+ \phi_E+ f_L= \phi_L+ \log\det(\omega_{\alpha\overline\beta})
\end{equation}
modulo a pluri-harmonic function on $U$. We see that
we are free to choose the trivialization of  $E$ and $L$ together with a coordinate system $(z_i)$ such that \eqref{mp63} becomes an equality by modifying 
the weights $\phi_A$.

\subsubsection{The approximation statement}
For $\delta,\ep>0$, Aubin-Yau theorem shows that the equation

\begin{equation}
\label{MAed}
(\om_{\delta}+\ddc \vp_{\delta,\ep})^n = (|s_{E}|^2+\ep^2)e^{\vp_{\delta,\ep}- f_L}  \om^n
\end{equation}
has a unique solution such that $\displaystyle \om_{\delta}+\ddc \vp_{\delta,\ep}$
is a smooth Kähler metric. In the two equations above, $|s_E|^2$ (resp. $|s_{E}|^2+\ep^2$) has to be interpreted as $\prod_i |s_i|^{2a_i}$ (resp. $\prod_i (|s_i|^2+\ep^2)^{a_i}$). We have the following convergence result.

\begin{prop}
\label{convv}
There exists a family of positive numbers $(\delta_{\ep})_{\ep>0}$ decreasing to $0$ when $\ep$ approaches zero such that
$$\lim_{\ep\to 0} ||\vp_{\delta_{\ep},\ep}-\vp||_{L^{\infty}(X)} =0.$$
\end{prop}

\begin{proof}
For now, let $\delta\in(0,1)$ be fixed. By \cite[Thm.~A]{GLZ}, one has 
\begin{equation}
\label{unif}
\limsup_{\ep\to 0} || \vp_{\delta,\ep}-\vp_{\delta, 0}||_{L^{\infty}(X)} =0.
\end{equation}

\noindent
Let $$\psi_{\delta}:=\frac{1}{1-\delta}\Big(\vp_{\delta,0}-n\log(1-\delta)\Big).$$
The above function satisfies the Monge-Ampère equation
\begin{equation}
\label{mapsi}
(\om_A+\frac{\delta}{1-\delta}\om_0+\ddc \psi_{\delta})^n=|s_E|^2e^{(1-\delta)\psi_{\delta}- f_L}\om^n.
\end{equation}
The $\om_A$-psh function $\hat\psi_{\delta}:=\psi_0+\frac{\delta}{1-\delta}\cdotp \inf_X \psi_0$ is a subsolution of \eqref{mapsi}.  Indeed, one has
\begin{align*}
\Big(\om_A+\frac{\delta}{1-\delta}\om_0+\ddc \hat \psi_{\delta}\Big)^n& \ge (\om_A+\ddc \psi_{0})^n\\
& = |s_E|^2e^{(1-\delta) \hat \psi_{\delta}- f_L}e^{\delta(\psi_0-\inf_X \psi_0)}\om^n\\
& \ge |s_E|^2e^{(1-\delta) \hat \psi_{\delta}- f_L}\om^n.
\end{align*}
Therefore, one gets $\hat \psi_{\delta}\le  \psi_{\delta}$, ie 
$$\psi_0 \le \psi_{\delta}-\frac{\delta}{1-\delta}\cdotp \inf_X \psi_0.$$
In particular, one finds a uniform lower bound $\psi_{\delta}\ge -C$ where $C>0$ is independent of $\delta$. Let $K:=\min\{0, \inf_{\delta}\inf_X \psi_{\delta}\}$ where the first infimum ranges over $\delta \in [0,e^{-1}]$ say. Using the same argument as above, one concludes that for any $0\le \eta \le \delta \le e^{-1}$, one has
$$\psi_{\eta}-\frac{K\eta}{1-\eta} \le \psi_{\delta}-\frac{K\delta}{1-\delta}.$$
That is, the family of $\om_A$-psh functions $(\psi_{\delta}-\frac{K\delta}{1-\delta})_{\delta>0}$ decreasing toward a bounded $\om_A$-psh $\tilde \psi_0$ when $\delta \downarrow 0$. The function $\tilde \psi_0$ satisfies the same Monge-Ampère equation \eqref{mapsi} as $\psi_0$ thanks to the continuity of the Monge-Ampère operator with respect to bounded decreasing sequences. Therefore, one has $\tilde \psi_0=\psi_0=\vp.$ 

 Now, the $\om_A$-psh function $\vp$ is continuous. Indeed, this is because $\om_A$ is the pull-back of a Hodge form on a (singular) space by a birational morphism, hence one can apply jointly \cite[Thm.~A]{EGZ} and \cite[Cor.~C]{CGZ}. All in all, Dini's theorem shows that the convergence $\psi_{\delta}-\frac{K\delta}{1-\delta}\to \vp$ is uniform. In particular, $\vp_{\delta, 0}$ converges uniformly to $\vp$ when $\delta \to 0$. The Proposition now follows from \eqref{unif} and a suitable diagonal process. 
\end{proof}




\subsection{Reduction to the case $c_1(L)\in H^2(X, \mathbb Z)$}
\label{sec:iteration} 

\bigskip
We fix an integer $p\ge 1$ such that $pE$ is integral and $pL$ is a line bundle. The first step in the algorithm which will follow consists in solving the equation
\begin{equation}
\label{IMA1}
(p\om_{\delta}+\ddc \vp_{1, \delta,\ep})^n = e^{\vp_{1, \delta, \ep}- f_L} (|s_{E}|^2+\ep^2)\om^n
\end{equation}
This is very similar to \eqref{MAed}. In particular, one can apply Proposition~\ref{convv} to show that there exists a family of numbers $(\delta_{\ep}^{(1)})_{\ep>0}$ decreasing to zero when $\ep \downarrow 0$ such that 
$$\limsup_{\ep\to 0}|| \vp_{1, \delta_{\ep}^{(1)},\ep}-\vp_{1,0,0}||_{L^{\infty}(X)}=0.$$
One sets $\vp_{1,\ep}:= \vp_{1, \delta_{\ep}^{(1)},\ep}$ and $\vp_1:=\vp_{1,0,0}$. Next, one solves the equation
\begin{equation}
\label{IMA1}
(p\om_{\delta}+\ddc \vp_{2, \delta,\ep})^n = e^{\vp_{2, \delta, \ep}-\frac{p-1}{p} \vp_{1,\ep}- f_L} (|s_{E}|^2+\ep^2)\om^n
\end{equation}
Proposition~\ref{convv} applies again verbatim to show that there exists a family of numbers $(\delta_{\ep}^{(2)})_{\ep>0}$ decreasing to zero when $\ep \downarrow 0$ such that 
$$\limsup_{\ep\to 0}|| \vp_{2, \delta_{\ep}^{(2)},\ep}-\vp_{2,0,0}||_{L^{\infty}(X)}=0.$$
We set $\vp_{2,\ep}:= \vp_{2, \delta_{\ep}^{(2)},\ep}$, $\vp_2:=\vp_{2,0,0}$ and repeat the procedure. The result is the following. 

\begin{prop}
\label{iteratiooon}
For each integer $m\ge 1$ there exist a family of positive reals $(\delta_{\ep}^{(m)})_{\ep>0}$ decreasing to zero and a family of smooth strictly $p\om_{\delta_{\ep}^{(m)}}$-psh functions $\vp_{m,\ep}$ such that 
\begin{equation}
\label{IMA}
(p\om_{\delta_{\ep}^{(m)}}+\ddc \vp_{m,\ep})^n = e^{\vp_{m, \ep}- \frac{p-1}{p}\vp_{m-1, \ep}- f_L} (|s_{E}|^2+\ep^2)\om^n
\end{equation}
and 
\begin{equation}
\label{convv2}
\limsup_{\ep\to 0}|| \vp_{m,\ep}-\vp_{m}||_{L^{\infty}(X)}=0.
\end{equation}
where $\vp_m$ are the unique $p\om_A$-psh bounded functions such that $\vp_0=0$ and 
\begin{equation}
\label{IMA2}
(p\om_A+\ddc \vp_{m})^n = |s_{E}|^2 e^{\vp_{m}- \frac{p-1}{p}\vp_{m-1}- f_L}\om^n
\end{equation}
\end{prop}

\bigskip

Thanks to \eqref{decomposition}, one gets for each integer $m\ge 1$ a decomposition 
\begin{equation}
\label{decomposition2}
K_X+L=A_{\delta_{\ep}^{(m)}}+E_{\delta_{\ep}^{(m)}}.
\end{equation}
and one can define the weights
\begin{equation}
\label{phiad3}
\phi_m:=p\phi_A+\vp_m, \quad \phi_{m,\ep}:=p\phi_{A_{\delta_{\ep}^{(m)}}}+\vp_{m,\ep}
\end{equation}
on $pA$ and $pA_{\delta_{\ep}^{(m)}}$ respectively, cf \eqref{phiad}. Let $\phi_E$ be a singular weight on $E$ such that $\ddc \phi_E=[E]$, and let $\phi_L$ a smooth weight on $L$ such that $\ddc \phi_L=i\Theta_{h_L}(L)$. \label{enditeration}The expressions $e^{\phi_m-\frac{p-1}p \phi_{m-1}+\phi_E-\phi_L}$ define a global volume form which we normalize (by adding a constant to $\phi_L$) such that
\begin{equation}
\label{norm}
\int_Xe^{\phi_m-\frac{p-1}p \phi_{m-1}+\phi_E-\phi_L}d\lambda=(A^n).
\end{equation} 
If follows from \eqref{IMA2} combined with the definition of $f_L$ cf. \eqref{mp33} that $\phi_m$ solves
 \begin{equation}
\label{IMA3}
(\ddc \phi_m)^n =e^{\phi_m-\frac{p-1}p \phi_{m-1}+\phi_E-\phi_L}d\lambda.
\end{equation}

\subsection{Convergence of the Ricci iteration}
The current $\om_m:=\ddc \phi_m=p\om_A+\ddc \vp_m$, satisfies the following twisted Kähler-Einstein-like equation:
\begin{equation}
\label{IKE}
\Ric \om_m = -\om_m+ \frac{p-1}{p}\om_{m-1}-[E]+  i\Theta_{h_{L}}(L) .
\end{equation}
Its behavior when $m\to +\infty$ is given by the following result. 
\begin{prop}\label{convricci}
When $m$ tends to $+\infty$, the current $\frac{1}{p} \om_m$ converges weakly to the (unique) twisted Kähler-Einstein metric $\omi\in c_1(A)$ solution of 
$$- \Ric \omi +i \Theta_{h_{L}}(L)= \omi+[E]$$
\end{prop}

\begin{rema}
\label{rmk:relation}
{\rm
We see that $\omi$ is equal to the Kähler-Einstein metric $\om_{\rm KE}$ of $(X,L, \phi_L)$ on $X\setminus E$. More precisely we have $\omke=\omi+[E]$. }
\end{rema}

\begin{proof}
Recall that $\om_m=p\om_A+\ddc \vp_m$ is solution of the Monge-Ampère equation
$$(p\om_A+\ddc \vp_m)^n=e^{\vp_m- \frac{p-1}{p}\vp_{m-1}}d\mu$$
where $d\mu=|s_E|^2  \cdotp \om^n$. 
We aim to show that for each $m\ge 2$, one has
\begin{equation}
\label{iter}
||\vp_m-\vp_{m-1}||_{L^{\infty}(X)}\le \frac{p-1}{p} ||\vp_{m-1}-\vp_{m-2}||_{L^{\infty}(X)}
\end{equation}
Let $C_{m}:= \sup_X (\vp_{m-1}-\vp_{m-2})$, and let $U_m=\{\vp_m > \vp_{m-1}+ \frac{p-1}{p} C_m\}$. An application of the comparison principle yields:
\begin{equation}\label{compais}
\int_{U_m}  e^{\vp_m- \frac{p-1}{p} \vp_{m-1}}d\mu\le \int_{U_m} e^{\vp_{m-1}- \frac{p-1}{p} \vp_{m-2}}d\mu . 
\end{equation}
On $U_m$, one has:
\begin{eqnarray*}
\vp_m- \frac{p-1}{p}\vp_{m-1}&>& \frac{1}{p}\vp_{m-1}+ \frac{p-1}{p} C_m\\
&=&(\vp_{m-1}- \frac{p-1}{p} \vp_{m-2})+ \frac{p-1}{p} [C_m-(\vp_{m-1}-\vp_{m-2})] \\
&\geq& \vp_{m-1}- \frac{p-1}{p} \vp_{m-2} .
\end{eqnarray*} 
Together with \eqref{compais}, we know that $U_m$ has measure zero with respect to $d\mu$, hence also with respect to $(p\om_A+\ddc \vp_m)^n$. 
By the domination principle, cf. e.g. \cite[Cor. 2.5]{BEGZ}, we see $U_m$ is empty, hence $\vp_m-\vp_{m-1} \le \frac{p-1}{p}  \sup_X (\vp_{m-1}-\vp_{m-2})$. 
Using an analogous argument, one can show that $\vp_m-\vp_{m-1} \ge \frac{p-1}{p}  \inf_X (\vp_{m-1}-\vp_{m-2})$, which proves \eqref{iter}. It follows by iteration that
$$||\vp_m-\vp_{m-1}||_{L^{\infty}(X)}\le (\frac{p-1}{p})^{m-1}||\vp_1||_{L^{\infty}(X)}$$
and therefore the sequence $(\vp_m)_{m\ge 1}$ converges uniformly to a $p\om_A$-psh function $\vp_{\infty}$. As Bedford-Taylor product is continuous with respect to uniform convergence, $\vp_{\infty}$ satisfies:
$$(p\om_A+\ddc \vp_{\infty})^n=e^{\frac 1 p \vp_{\infty}}d\mu$$
which proves the proposition.
\end{proof}

\subsection{Convergence of the Bergman kernel iteration }
\label{sec:conv}
In this paragraph we fix an integer $m\ge 1$ and we prove that the twisted Kähler-Einstein metric $\om_m$ is the weak limit of iterated Bergman kernels.

\noindent Consider the line bundle 
\begin{equation}\label{mp35}
L_p:= (p-1)(K_X+ L)+ L,
\end{equation}
where $p$ is a positive integer such that $pL$ is a line bundle and $pE$ has integer coefficients.   
We 
recall that the triple $(X,L,h_L)$ satisfies the following.
\begin{enumerate}
\item[$\bullet$] $K_X+L=A+E$ is a Zariski decomposition of the big line bundle $K_X+L$.
\item[$\bullet$] The hermitian metric $h_L=e^{-\phi_L}$ on the $\Q$-line bundle $L$ is a smooth and has semipositive curvature, i.e. $dd^c \phi_L \ge 0$.
\end{enumerate}
We then endow $L_p$ with the metric given by the weights
\begin{equation}\label{mp36}
\tau_m:= (p-1)\left(\frac{\phi_{m-1}}p+\phi_{E}\right)+ \phi_L
\end{equation}
where $\phi_{m-1}$ is the weight corresponding to the metric on $pA$ defined by \eqref{phiad3}, $\phi_{E}$ is a singular weight on $\mathcal O_X(E)$ such that $\ddc \phi_{E}=[E]$ and finally $\phi_L$ is the smooth metric on $L$ satisfying $\ddc \phi_L= \Theta_{h_L}(L)$ such that we have the equality \eqref{IMA3}.
\smallskip

We write 
$$(\ell+1)(K_X+L_p)=K_X+\ell (K_X+L_p)+ L_p$$
and then we can define a singular metric $h_{\ell}$ 
on the line bundle $\ell (K_X+L_p)$ by induction on $\ell$ in the following manner: 
\begin{equation}
 \label{iteration} 
h_{\ell+1}:=K_{\ell+1}^{-1} 
\end{equation}
where 
$$K_{\ell+1}:=K(X,(\ell+1)(K_X+L_p), h_{\ell }\cdotp e^{-\tau_m})$$ 
is the Bergman kernel of $(\ell+1)(K_X+L_p)$ endowed with the metric above.
Of course that it depends on $m$ and $p$, even if our notation
does not reflects this.
\medskip

\noindent In the current subsection we are aiming at the following result, from which Theorem~\ref{thmaa} will follow easily. 

\begin{theo}\label{mainpropsubsect16}
Under the assumptions above, the renormalized Bergman kernels $ \left(n!^{\ell}\ell!^{-n}K_{\ell}\right)^{1/\ell}$ converge to $e^{\phi_m+p\phi_E}$ as $\ell\to +\infty$.
\end{theo}

Prior to the proof of this result we are making a few preliminary remarks
concerning the singularities of $h_{\ell}$. Since $K_X+L=A+E$ is a Zariski decomposition, one knows that if $p$ is divisible enough, the multiplication by $s_{pE}$ induces an isomorphism $H^0(X,pA) \to H^0(X,p(K_X+L))$. Therefore all sections $s\in H^0(X,\ell (K_X+L_p))$ vanish along $pE$ at order at least $\ell $. Since $A$ is semi-ample, there exists a section whose
vanishing order along $pE$ is exactly $\ell$. A quick induction shows that every section $s\in H^0(X,(\ell+1)(K_X+L_p))$ is square integrable with respect to $h_{\ell}\cdotp e^{-\tau_p}$. \\

\noindent The MA equation for $\phi_m$ is as follows
 \begin{equation}
\label{mp37}
(\ddc \phi_m)^n =e^{\phi_m+ \phi_E-\phi_L-\frac{p-1}{p}\phi_{m-1}}d\lambda
\end{equation}
given \eqref{IMA3}. The solution $\phi_m$ is not regular enough for what is needed in the arguments to follow, so we also consider the
approximation obtained in Proposition \ref{iteratiooon} for which we have
\begin{equation}\label{mp38}
(\ddc \phi_{m, \ep})^n =e^{\phi_{m, \ep}-\frac{p-1}{p}\phi_{m-1, \ep}-\phi_L}(\ep^2e^{\rho_E}+ e^{\phi_E})
d\lambda
\end{equation}
where $e^{-\rho_E}$ is the smooth metric on $E$
we fixed in section \ref{sec:approx}.
\smallskip

\noindent The weights
\begin{equation}\label{mp39}
\ell\left(\phi_{m, \ep}+ p\phi_{E_\ep}\right)
\end{equation}
are defining a metric on $\ell(K_X+ L_p)$ and therefore the quantity    
\begin{equation}\label{defninf}
C_{\ell}:=\inf_X \frac{K_{\ell}}{|s_{E_\ep}|^{2\{\ell p\}}e^{\ell\left(\phi_{m, \ep}+ p\phi_{E, \ep}\right)}}
\end{equation}
is a strictly positive real number, cf. the discussion above.
\medskip

\noindent  The proof of Theorem \ref{thmaa} relies heavily on the
following statement. 

\begin{prop}
\label{keyprop1}
For every $m$ fixed, there exists $\kappa_{\ep, \ell} >0$ such that: 
$$C_{\ell} \ge \kappa_{\ep, \ell} \cdotp \frac{\ell^n}{n !} \cdotp C_{\ell-1}$$
and 
$$\lim_{\ep \to 0 } \lim_{\ell \to +\infty}\left(\prod_{k=1}^{\ell }\kappa_{\ep,k}\right)^{1/\ell}=1 .$$ 
\end{prop}

\begin{proof}
We have organized our arguments in four main steps. \\

\noindent
\textbf{Step 1. Choice of an appropriate local section $u$.}
\vspace{2mm}

\noindent Let $x_0\in X$ be an arbitrary point.
Let $x_0\in U\subset X$ be an open subset of $X$ such that the restriction to $U$ of all our bundles (i.e. $pL, pE...$) is trivial. We consider the local weights $\rho_i$ for $h_i$ cf. Section~\ref{not} such that $\rho_i(x_0)= 0$ for all $i$.
We take a coordinate system $(z_i)_{i=1,\dots, n}$ on $U$, centered at $x_0$, and we assume that \eqref{mp63} holds. All the local computations to follow are done with respect to this data.
\smallskip

\noindent We introduce the quadratic function
\begin{equation}\label{mp339}
  h(z):= \phi_{m, \ep}(x_0)+ 
  + 2\sum_i\partial_i\left(\phi_{m, \ep}\right)(x_0)z_i
 + 4\sum_{i,j}\partial^2_{i,j}\left(\phi_{m, \ep}\right)(x_0)z_iz_j
\end{equation}
and the holomorphic section of $(1+\ell)(K_X+ L_p)|_U$
\begin{equation}\label{mp40}
u:= \prod f_{i}^{\lceil p(1+ \ell)a_i^\ep\rceil}\cdotp e^{\frac{ 1+\ell}{2}h}dz\otimes e_{K_X+ L_p}^{\otimes \ell}\otimes e_{L_p} 
\end{equation}
written as $(n,0)$-form with values in $\ell(K_X+ L_p)+ L_p|_U$. Note that $a_i^\ep> a_i$ for each index $i$.
\smallskip

\noindent We denote by $h_{\ep}$ the metric on $\ell(K_X+ L_p)+ L_p$ defined by the weights
\begin{equation}\label{mp66}
\ell\left(\phi_{m, \ep}+ p\phi_{E_\ep}\right)+ \tau_m+ \log|s_{E_\ep}|^{2\{\ell p\}}
\end{equation}
where we recall that $\displaystyle \tau_m= (p-1)\left(\frac{\phi_{m-1}}p+\phi_{E}\right)+ \phi_L$ was introduced in \eqref{mp36}.
The measure induced by the point-wise norm of $u$ with respect to the metric \eqref{mp66} is equal to
\begin{equation}\label{mp41}
  |u|^2_{h_\ep}= 
  e^{(1+ \ell)\Re(h)}\cdotp e^{-\ell\left(\phi_{m, \ep}+ p\phi_{E,\ep}\right)- \tau_{m}}\prod |f_i|^{2\lceil p(1+ \ell)a_i^\ep\rceil}\frac{d\lambda}{|s_{E_\ep}|^{2\{\ell p\}}}
\end{equation}
and it can be reorganized as follows
\begin{equation}\label{mp42}
  |u|^2_{h_\ep}= e^{(1+ \ell)\left(\Re(h)- \phi_{m, \ep})\right)}
  \cdotp \prod |f_i|^{2\mu_i}\cdotp e^{\{\ell pa^\ep_i\}\rho_i}e^{\phi_{m, \ep}-\frac{p-1}{p}\phi_{m-1}+ \phi_E-\phi_L}d\lambda
\end{equation}
where $\mu_i:= \lceil p(1+ \ell)a_i^\ep\rceil- p\ell a_i^\ep- p a_i- \{\ell pa^\ep_i\}$. 

We therefore have the point-wise inequality
\begin{equation}\label{mp43}
|u|^2_{h_\ep}\leq \gamma_\ep F e^{(1+ \ell)\left(\Re(h)- \phi_{m, \ep}\right)}(dd^c\phi_{m, \ep})^n
\end{equation}
on the set $U$, where the positive function $F$ and the constant $\gamma_\ep$
are as follows.
\begin{enumerate}

\item[(i)] We define $\displaystyle \gamma_\ep:= \sup_U e^{\frac{p-1}{p}|\phi_{m-1, \ep}- \phi_{m-1}|}$. By Proposition \ref{iteratiooon}, we have $\gamma_\ep\to 0$.

\item[(ii)] Let $F:= \prod |f_i|^{2\mu_i}\cdotp e^{\{\ell pa^\ep_i\}\rho_i}$ be the function corresponding to the product in \eqref{mp42}. We have
  $$F= \prod |f_i|^{2(\lceil p(1+ \ell)\ep c_i\rceil- \lceil p\ell\ep c_i\rceil)}\cdotp e^{\{\ell pa^\ep_i\}\rho_i}$$
hence $F(0)\leq 1$, and $\sup_U(F)- 1$ is smaller than the diameter of $U$ multiplied with a bounded constant.    
\end{enumerate}

\medskip

\noindent\textbf{Step 2. Estimate of the $L^2$ norm of $u$.}
\vspace{2mm}

\noindent
Let $B(r_{\ep})$ be the Euclidean ball centered at $x_0$ of radius $r_{\ep}$ with respect to the fixed Kähler metric $\om$. We have the following inequalities, which will be proved by a direct computation at the end of
this section.

\begin{claim}
\label{cla}
 For every $\ep > 0$ fixed, there exists a radius $r_\ep$ and a sequence $a_\ell$ converging to $0$
(independent of $x_0\in X$), such that
\begin{equation}\label{totall2}
\frac{(\ell+1)^n}{n !} \int_{B(r_{\ep})} F e^{-(\ell+1) (\phi_{m,\ep}-\mathrm{Re}(h))}(\ddc \phi_{m,\ep})^n \leq 1+ a_\ell \qquad\text{for every } \ell\in \mathbb{N},
\end{equation}
and
\begin{equation}\label{totall21}
 (\ell+1)^{n} \int_{B(r_{\ep})\setminus B(\frac{r_{\ep}}{2})} F e^{-(\ell+1) (\phi_{m,\ep}-\mathrm{Re}(h))}(\ddc \phi_{m,\ep})^n 
\leq  a_\ell \qquad\text{for every } \ell\in \mathbb{N} .
\end{equation}
\end{claim}
An important point of the claim is that the sequence $\{a_\ell\}$ is independent of $x_0$.
\smallskip

\noindent Combined with the inequality \eqref{mp43} we obtain
\begin{equation}\label{mp44}
\frac{(\ell+1)^n}{n !} \int_{B(r_{\ep})} |u|^2_{h_\ep}\leq 1+ a_\ell \qquad\text{for every } \ell\in \mathbb{N},
\end{equation}
and
\begin{equation}\label{mp45}
 (\ell+1)^{n} \int_{B(r_{\ep})\setminus B(\frac{r_{\ep}}{2})}  |u|^2_{h_\ep}
\leq  a_\ell \qquad\text{for every } \ell\in \mathbb{N}.
\end{equation}

\noindent
\textbf{Step 3. Construction of a global section.}

\vspace{2mm}

\noindent
For every $\ep >0$ fixed,  we will construct in this step a section $v_{\ell,\ep} \in H^0 (X, (\ell +1) p (K_X +L))$ such that $v_{\ell,\ep} (x_0)=u (x_0)$ together with an estimate for its $L^2$ norm
\begin{equation}\label{globalsect}
 \frac{(\ell+1)^n}{n !} \cdot \int_X |v_{\ell, \ep}|^2 _{h_{\ell}}e^{-\tau_m}.
\end{equation}
uniform with respect to the point $x_0\in X$.

\medskip

Let $\rho$ be a smooth function on $X\setminus \{x_0\}$ which equals
$n \log |x- x_0|^2$ near $x_0$.
For every $\ep>0$, we can find a cut-off function $\chi_{\ep}$ for $B(r_{\ep})$, namely $\chi_{\ep}\equiv 1$ on $B(r_{\ep}/2)$ and $\chi_\ep \equiv 0$ on 
$X\ssm B(r_{\ep})$ such that
\begin{equation}\label{normcontew}
e^{-\rho}|\dbar \chi_\ep|^2_{\om_{m,\ep}} \leq M_\ep \qquad\text{on }
B(r_{\ep})\smallsetminus B(r_{\ep}/2)
\end{equation}
 for some constant $M_\ep$ independent of $x_0\in X$.
One can easily check that for $\ell$ large enough (depending on $\ep$), we have
\begin{equation}\label{curterm}
  \ell dd^c(\phi_{m,\ep}+ p\phi_{E,\ep})+ dd^c\tau_{m, \ep}+
  \ddc \big(\rho+ \log|s_{E_\ep}|^{2\{(1+\ell)p\}}\big) \geq dd^c\phi_{m,\ep} 
\end{equation}
on $X$, since $\displaystyle dd^c\phi_{m,\ep}= \om_{m, \ep}$ is a K\"ahler metric for each $m,\ep$. 
Thanks to \eqref{mp44} we have
$$ \frac{(\ell+1)^n}{n !} \int_X |\chi_\ep  u|^2_{h_\ep} \leq 1+ a_\ell .$$
By the inequality \eqref{mp45} and the construction of $\chi_\ep$, we have
$$
\int_X |\bar \d (\chi_\ep u)|^2 _{{h_\ep}, \om_{m,\ep}} \leq \int_{B(r_{\ep})\setminus B(\frac{r_{\ep}}{2})}   e^{-\rho}|\dbar \chi_\ep|^2_{\om_{m,\ep}}   \cdotp   e^{-(\ell+1)(\phi_{m,\ep} -\mathrm{Re}(h))} (\ddc \phi_{m,\ep})^n .
$$
Together with \eqref{mp45} and \eqref{normcontew}, we get 
$$(\ell+1)^n
\int_X |\bar \d (\chi_\ell u)|^2 _{{h_\ep},\om_{m,\ep}} e^{-\rho}\leq a_\ell \cdot M_\ep .$$
Thanks to \eqref{curterm}, one can solve the $\bar \d$-equation and apply Hörmander estimates (see e.g. \cite[Cor.~14.3 on p.86]{BDIP}) for $\ell$ large enough (independent of $x$) to the $\bar \d$-closed form
$$\bar \d (\chi_\ep u) \in \mathcal C^{\infty}(X,\Lambda^{n,1}T_X^*\otimes E)$$
where $E= \ell(K_X+L_p)+ L_p$ is endowed with the hermitian metric $h_{\ell,\ep}$. This yields a global, smooth section $u_{\ell,\ep}$ of $\Lambda^{n,0}T_X^*\otimes E=(\ell+1)(K_X+L_p)$ such that
$$\dbar u_{\ell,\ep}= \bar \d (\chi_\ep u) \quad \mbox{and} \quad (\ell+1)^{n}\int_{X} |u_{\ell,\ep} |^2_{h_\ep} \leq  a_\ell \cdot M_\ep. $$
Because of the non-integrability of $e^{-\rho}$ at $x_0$, one has $u_\ell (x_0)=0$. 
As a consequence, the section 
$v_{\ell,\ep}:=\chi_\ep u-u_{\ell,\ep} \in H^0 (X, (\ell +1) (K_X +L_p))$ satisfies the inequality

\begin{equation}\label{mp56}
 \frac{(\ell+1)^n}{n !} \cdot \int_X |v_{\ell, \ep}|^2_{h_\ep} \leq 1+ a_\ell 
\end{equation}
and we also have $v_{\ell,\ep}(x_0)= u(x_0)$. By definition of $C_\ell$ we have
\begin{equation}\label{mp57}
\frac{(\ell+1)^n}{n !} \cdot \int_X |v_\ell|^2_{h_\ell} e^{- \tau_{m}} \leq \gamma_\ep\frac{1+ a_\ell}{C_\ell}. 
\end{equation}

\noindent
\textbf{Step 4. Conclusion.}
\vspace{2mm}

\noindent
Thanks to \eqref{mp57} we obtain that the following inequality
$$K_{\ell+1 } \ge \frac{ C_{\ell}}{\gamma_\ep(1+a_\ell) } \frac{(\ell+1)^n}{n !}\cdotp e^{(\ell+1) (\phi_{m,\ep}+p\phi_{E_{\ep}})}|s_{E_\ep}|^{2\{(1+\ell)p\}}$$
at $x_0$. Therefore $$C_{\ell+1}\ge \kappa_{\ep, \ell}  \cdot \frac{(\ell+1)^n}{n !} \cdot C_{\ell} ,$$
where $\kappa_{\ep, \ell}= ( \gamma_{\ep}  (1+a_\ell))^{-1} $. Although the sequence $\{a_\ell\}$ depends on $\ep$
we have 
$$\lim_{\ell \to +\infty}\left(\prod_{k=1}^{\ell }\kappa_{\ep,k}\right)^{1/\ell}=\gamma_{\ep} $$
since it tends to $0$.
Therefore
$$\lim_{\ep \to 0 } \lim_{\ell \to +\infty}\left(\prod_{k=1}^{\ell }\kappa_{\ep,k}\right)^{1/\ell}=1, $$
and Proposition~\ref{keyprop1} is proved.
\end{proof}

\noindent
It remains to prove Claim~\ref{cla} stated in Step 3. 
\begin{proof}[Proof of Claim~\ref{cla}]
\label{proof}

Up to replacing $\phi_{m,\ep}$ by $ \phi_{m,\ep}-\mathrm{Re}(h)$ (which does not change the metric $\ddc \phi_{m,\ep}$), one can assume that $\phi_{m,\ep}$ has no polyharmonic terms of order two or less in its expansion near $x_0$, and $F(x_0)=1$. With respect to local coordinates $(z_i)$ centered at $x$, one has 

$$\phi_{m,\ep}= \sum_{j,k} a_{j,k} z_j \bar z_k + R(z)$$ 
where $R(z)=O(|z|^3)$ and the matrix $A=(a_{j,k})$ is positive definite. These quantities are depending on $\ep$ but the important point is that, when $\ep>0$ is fixed, one can find a constant $C_{\ep}>0$ independent of the chosen point $x$ such that 
$$|R(z)| \le C_{\ep} |z|^3, \quad  C_{\ep}^{-1} I_n < A <  C_{\ep} I_n.$$
The constant $C_{\ep}$ can be chosen to be commensurable to $\sup_X(|\om_{m,\ep}|_\om+ |\nabla^{\om}\om_{m,\ep}|)$.  
\medskip

After the change of variable $ w:=\sqrt{\ell+1}\sqrt A \cdotp z$  and up to increasing the constant $C_{\ep}$ a little, the integral we have to bound is dominated by
$$\int_{|w|^2 \le (\ell+1) r_{\ep}^2} e^{-|w|^2(1-C_{\ep}\ell^{-1/2}|w|)} (1+C_{\ep}\ell^{-1/2}|w|) \cdotp (dd^c |w|^2)^n/n!.$$
Now, if one chooses $r_{\ep}<1/(2C_{\ep})$, one sees that the integrand is less than $2 e^{-|w|^2/2}$, hence one can apply the dominated convergence theorem to conclude that our integral is asymptotically dominated by
$$\int_{\mathbb C^n} e^{-|w|^2} \cdotp (dd^c |w|^2)^n/n!=\frac{1}{\pi^n}\int_{\mathbb C^n}e^{-|w|^2} idw_1\wedge d\bar w_1 \wedge \cdots \wedge idw_n\wedge d\bar w_n=1 $$
which concludes the proof of \eqref{totall2}.

\medskip

As for \eqref{totall21},  the same change of variable reduces our integral to 
$$ \int_{\frac 1 4 (\ell+1) r_{\ep}^2 \le |w|^2 \le (\ell+1) r_{\ep}^2} e^{-|w|^2(1-C_{\ep}\ell^{-1/2}|w|)} (1+C_{\ep}\ell^{-1/2}|w|) \cdotp (dd^c |w|^2)^n/n!$$
which, up to increasing $C_{\ep}$, is dominated by
$$C_{\ep} \int_{\frac 1 4 (\ell+1) r_{\ep}^2 \le |w|^2\le (\ell+1)r_{\ep}^2} e^{-\ell r_{\ep}^2/8}\cdotp (dd^c |w|^2)^n=O(\ell^{n}e^{-\ell r_{\ep}^2/8}).$$
The estimate \eqref{totall21} follows.
\end{proof}

\noindent 
We will also need an integral upper estimate of $K_{\ell}$; it follows easily from the definition of the Bergman kernel.
\begin{prop}
\label{prop:int}
One has the following upper bound:
$$\limsup_{\ell \to +\infty} \int_X n! (\ell !^{-n} K_{\ell})^{1/\ell} e^{-\tau_p} \le(pA)^n $$
\end{prop}

\begin{proof}
First, observe that $(\ell !^{-n} K_{\ell})^{1/\ell}e^{-\tau_p}$ is a volume form, so that the claim is licit. Let $(u_1, \ldots , u_{N_{\ell}})$ be an orthonormal basis of $H^{0}(X, \ell (K_X+L_p))$ with respect to the Bergman $L^2$ metric $$K_{\ell-1}^{-1}e^{-\tau_p}$$ Since $\ell p(K_X+L)=\ell pA+\ell pE$ is the Zariski decomposition of $\ell p(K_X+ L)$, every (pluri)-section is $L^2$ with respect to the Bergman metric. In particular we have 
\begin{equation}
\label{RR}
N_{\ell}:=\dim H^0(X,\ell p A) = \frac{(pA)^n}{n!}\cdotp \ell^{n}  (1+ O(\ell^{-1}))
\end{equation}
by Riemann-Roch formula. One has
\begin{eqnarray*}
\int_{X}K_{\ell}\cdotp K_{\ell-1}^{-1} e^{-\tau_p} = N_{\ell}
\end{eqnarray*}
Therefore, applying Hölder's inequality with $p=\ell$ and $q=\frac{\ell}{\ell-1}$, one gets
\begin{eqnarray*}
  \int_X K_{\ell}^{\frac 1{\ell}}\cdotp e^{-\tau_p} &\le&  \left(\int_{X}K_{\ell}
                                                          \cdotp K_{\ell-1}^{-1}
                                                          e^{-\tau_p}\right)^{\frac{1}{\ell}} \cdotp \left(\int_X K_{\ell-1}^{\frac{1}{\ell-1}}e^{-\tau_p}\right)^{\frac{\ell-1}{\ell}} \\
& \le & N_{\ell}^{1/\ell} \cdotp \left(\int_X K_{\ell-1}^{\frac{1}{\ell-1}}e^{-\tau_p}\right)^{\frac{\ell-1}{\ell}}
\end{eqnarray*}
By induction, one gets: 
$$\int_X K_{\ell}^{\frac 1{\ell}}\cdotp e^{-\tau_p} \le \left(\prod_{i=1}^{\ell}N_i\right)^{\frac 1 {\ell}}$$
Now, thanks to \eqref{RR}, the right-hand side of this inequality is equal to $$ \frac{(pA)^n}{n!}\cdotp (\ell!)^{\frac{n}{\ell}} \left(1+O\left(\frac{\log \ell}{\ell}\right)\right)$$
which concludes the proof of the proposition.\\
\end{proof}

\noindent
After all these preliminary statements
we can prove the main result of this subsection.
\begin{proof}[Proof of Theorem~\ref{mainpropsubsect16}]
By Proposition \ref{prop:int} combined with Jensen inequality, $\{\left(\ell!^{-n}K_{\ell}\right)^{1/\ell}\}_{\ell=1}^{+\infty}$ is a family of upper bounded psh weights.
Therefore, to prove the Theorem, it is sufficient to prove that any convergent subsquence of $\{n!\left(\ell!^{-n}K_{\ell}\right)^{1/\ell}\}_{\ell=1}^{+\infty}$ converges to $e^{\phi_m + p\phi_E}$.

\noindent
Let $\{n!\left(\ell_s !^{-n}K_{\ell_s}\right)^{1/\ell_s}\}_{s=1}^{+\infty}$ be a convergent subsequence of $\{n!\left(\ell!^{-n}K_{\ell}\right)^{1/\ell}\}_{\ell=1}^{+\infty}$
and let $\Gamma$ be the limit. 
Thanks to Proposition \ref{keyprop1}, one infers:
$$n!\left(\ell!^{-n}K_{\ell}\right)^{1/\ell} \geq  \left(\prod_{k=1}^{\ell }\kappa_{\ep,k}\right)^{1/\ell} \cdot  |s_{E_\ep}|^{\frac{2\{\ell p\}}{\ell}}e^{\phi_{m,\ep} + p\phi_{E_\ep}}$$
Letting $\ell$ tend to $+\infty$, and then letting $\ep$ tend to zero, item~\eqref{convv2} in Proposition~\ref{iteration}
and Proposition \ref{keyprop1} yield:
\begin{equation}
\label{liminf}
\liminf_{\ell \to +\infty} n!\left(\ell!^{-n}K_{\ell,m}\right)^{1/\ell}  \geq   e^{\phi_{m}+p \phi_E} .
\end{equation}
Therefore, we have 
\begin{equation}
\label{limonedir}
\Gamma \geq e^{\phi_{m}+p \phi_E}   .
\end{equation}

\noindent
Note that 
$$\int_X \Gamma  e^{-(p-1)\left(\frac{\phi_{m-1}}p+\phi_{E}\right)}e^{-\phi_L}  
=\lim_{\ell\rightarrow +\infty} \int_X n!(\ell_s !^{-n} K_{\ell_s, m})^{1/\ell_s} e^{-(p-1)\left(\frac{\phi_{m-1}}p+\phi_{E}\right)}e^{-\phi_L} .$$
Combining this with Proposition \ref{prop:int}, we get 
$$\int_X \Gamma  e^{-(p-1)\left(\frac{\phi_{m-1}}p+\phi_{E}\right)} e^{-\phi_L} \leq  (pA)^n  =\int_X e^{\phi_m +p\phi_E} \cdot e^{-(p-1)\left(\frac{\phi_{m-1}}p+\phi_{E}\right)}e^{-\phi_L} ,$$
where the last equality comes from \eqref{norm}.
Together with \eqref{limonedir}, we get $\Gamma = e^{\phi_m + p\phi_E} $ and the theorem is proved.
\end{proof}

\begin{rema}
\label{global}
{\rm
The convergence $(\ell!^{-n}K_{\ell})^{1/\ell} \to \frac{e^{\phi_m +p\phi_E}}{n!}$ has been proved for a fixed fiber $X_y$, but it readily implies convergence in $L^1_{\rm loc}(X°)$. Indeed, as $(A_y^n)$ is independent of $y\in Y°$, Proposition~\ref{prop:int} coupled with Jensen inequality show that the weights of the metric $(\ell!^{-n}K_{\ell})^{-1/\ell}$ are uniformly bounded above locally near 
$X\ssm X_0$, hence the pointwise convergence almost everywhere on $X°$ implies convergence in $L^1_{\rm loc}(X°)$.}
\end{rema}

\noindent
Now, we can finally give the proof of Theorem~\ref{thmaa}.

\begin{proof}[Proof of Theorem~\ref{thmaa}]
\label{proofthmaa}
Thanks to the reduction steps, we can suppose that on the fibers over $y\in Y^0$, we have a Zariski decomposition
$$(K_{X/Y} +L) |_{X_y}  =A_y +E_y ,$$
where $A$ is semi-ample and big, $E_y$ has snc support and $h_{L}$ is smooth with semipositive curvature. %

\medskip

We first prove by induction that for every $m\in\mathbb{N}$, $\phi_m+p\phi_E$ is a psh weight on $X°$.  

For $m=1$:  We get a sequence of metrics $(h_{\ell,1})_{\ell \ge 1}$ on $p\ell(K_X+L)$ defined by \eqref{iteration}. 
Recall that $\phi_0=\phi_A$ is the weight of $\omega_A$.
Thanks to \cite[Thm.~0.1]{BP}, $h_{1,1}$ has positive curvature on the total space $X$.
We suppose by induction that $h_{\ell,1}$ has positive curvature.
Then $h_{\ell, 1} \cdot e^{-(p-1)\left(\frac{\phi_{A}}p+\phi_{E}\right)}e^{-\phi_L}$
has also positive curvature. By applying \cite[Thm.~0.1]{BP} again, 
$h_{\ell+1,1}$ has positive curvature on the total space $X$.
As a consequence, $h_{\ell,1}$ has positive curvature on the total space $X$ for all $\ell$.
Together with Theorem \ref{mainpropsubsect16}, the limit $\phi_1+p\phi_E$ is a psh weight on $X^0$.

Then we apply the same process again to $m=2$, and get a sequence of metrics $(h_{\ell, 2})_{\ell \ge 1}$ on $p\ell(K_X+L)$ with positive curvature, and therefore the limit $\phi_2+p\phi_E$ is psh. By induction on $m$, we know that $\phi_m+p\phi_E$ is psh for any $m\ge 1$.

\medskip
We can now prove the Theorem. Thanks to Proposition~\ref{convricci} and Remark~\ref{rmk:relation}, $\phi_m+p\phi_E$ converges to the relative Kähler-Einstein metric of $(X°,L,\phi_L)$. As  $\phi_m+p\phi_E$ has positive curvature,  the relative Kähler-Einstein metric has also positive curvature. 
\end{proof}

\subsection{The orbifold case}
\label{orbifold}
In this paragraph, we would like to discuss an extension of Theorem~\ref{thmaa} to a particular case where the weight $\phi_L$ could have positive
Lelong numbers. More precisely, let $p:X\to Y$ be a Kähler fiber space and let $B=\sum_{i\in I}\left(1-\frac 1{m_i}\right)B_i$ be an effective divisor such that:

$\bullet $ The $\mathbb Q$-line bundle $K_{X/Y}+B$ is $p$-ample. 

$\bullet $ For $y\in Y$ generic, $B|_{X_y}$ has snc support.  

$\bullet$ The numbers $m_i \ge 1$ are integers such that $(m_i,m_j)=1$ whenever $B_i \cap B_j \neq \emptyset$.

\noindent
The same strategy as for the proof of Theorem~\ref{thmaa} applies,
modulo the fact that $\phi_{m,\ep}$ from \eqref{mp38} are to be replaced by their
orbifold counterparts, so that we have
\begin{equation}
\label{hg38}
(\ddc \phi_{m, \ep})^n =e^{\phi_{m, \ep}-\frac{p-1}{p}\phi_{m-1, \ep}-\phi_B}(\ep^2e^{\rho_E}+ e^{\phi_E})
d\lambda
\end{equation}
where $\phi_B$ is the canonical \textit{singular} weight on the $\mathbb Q$-line bundle $\mathcal O_X(B)$. The current $\ddc\phi_{m, \ep}$ defines an \textit{orbifold Kähler metric}, that is, its pull back to local uniformizing charts near the support of $B$ becomes a genuine Kähler metric. In this setting, a  new problem arise (due to the presence of singularities): the peak sections from Proposition~\ref{keyprop1} have to be replaced by orbifold peak sections.
One way to bypass this is to use the orbifold Bergman kernel expansion due to Ross-Thomas \cite{RT11}, see also Dai-Liu-Ma \cite{DLM12}. Instead of considering the Bergman kernels $K_{\ell}$ on $\ell(K_{X/Y}+B)$, one considers suitable linear combinations of the Bergman kernels of the form $\sum_{\alpha=0}^{m_i-1} K_{\ell+\alpha}$. These combinations, unlike $K_{\ell}$ alone, turn out to admit the same expansion as in the smooth case at order zero, when $\ell \to +\infty$. This is where the assumption on the arithmetic relation between the $m_i's$ is important. The rest of the proof of Theorem~\ref{thmaa} can be applied almost without any change to conclude. 

It is likely that combining the ideas above with the techniques in \cite{RT11} may allow us to weaken the assumption that  $K_{X/Y}+B$ is $p$-ample and only assume that $K_{X/Y}+B$ is $p$-big admits a relative Zariski decomposition on $X$.

\section{The case of intermediate Kodaira dimension}\label{ssec:intkod}

\subsection{Iitaka fibration and associated K\"ahler-Einstein metric}

We will first consider the absolute case. Let $X$ be a compact K\"ahler manifold and let $B$ be a $\mathbb{Q}$-effective divisor such that the pair $(X, B)$ is klt and such that $\kappa(K_X+B) >0$.
Let $Z$ be the canonical model of $(X, B)$. We consider
$f: X \dashrightarrow Z$ the Iitaka fibration induced by the linear system $|m(K_X+B)|$ for $m$ large and divisible enough.
Thanks to \cite{BCHM} in the projective case and \cite{Fuj13} for the K\"ahler case the space $Z$ is normal.
After desingularisation $f$ induces a fibration between two compact K\"ahler manifolds.
For simplicity we will denote the new map by $f: X\to Z$.  
In general the torsion-free sheaf $f_\star (m(K_{X/Z}+B))$ is not locally free.
This is the case for the reflexive hull $(f_\star (m(K_{X/Z}+B)))^{\star\star}$.
\medskip

\noindent We now recall the definition of the Narashimhan-Simha metric on the line bundle $(f_\star (m(K_{X/Z}+B)))^{\star\star}$.

\begin{defin}
Let $Z_0 \subset Z$ be a locus such that $f$ is smooth over $Z_0$ and $B|_{X_z}$ is klt for every $z\in Z_0$. 
Let $z\in Z_0$ and let $s\in (f_\star (m(K_{X/Z}+B)))_z = H^0 (X_z, m K_{X_z} +m B)$. We define the Narashimhan-Simha metric
$$ \|s\|^2 _{h_m} := (\int_{X_z} |s|^{\frac{2}{m}} _{h_B})^m ,$$
where $h_B$ is the canonical singular hermitian metric with respect to the divisor $B$.
Thanks to \cite{BP}, $h_m$ can be canonically extended as a possible singular metric on 
$$(Z, (f_\star (m(K_{X/Z}+B)))^{\star\star}) .$$
We call it the $m$-th Narashimhan-Simha metric.
\end{defin}

\begin{rema}\label{indepofm}
We can easily check that the weight of $h_m$ is locally integrable 
over the locus where $f_\star (m(K_{X/Z}+B))$ is locally free.  
Moreover, the pair $(Z, \frac{1}{m} (f_\star (m(K_{X/Z}+B)))^{\star\star}, h_m ^{\frac{1}{m}})$ is independent of the choice of $m$, namely
for any two $m_1, m_2$ large and sufficiently divisible, we have an isometry 
$$(\frac{1}{m_1} (f_\star (m_1 (K_{X/Z}+B)))^{\star\star}, h_{m_1} ^{\frac{1}{m_1}} )\equiv_\mathbb{Q} (\frac{1}{m_2} (f_\star (m_2 (K_{X/Z}+B)))^{\star\star}, h_{m_2} ^{\frac{1}{m_2}} ).$$
\end{rema}
\medskip

By construction, the pair $(Z, \frac{1}{m} (f_\star (m(K_{X/Z}+B)))^{\star\star}, h_m ^{\frac{1}{m}})$ satisfies \ref{dbig} and \ref{fg} 
However, it does not satisfy in general \ref{nef}, roughly because of the codimension two subsets of the base $Z$ whose $f$-inverse image have codimension one.
This situation can be improved by a trick due to \cite{Vieh83} which we now recall.
By Hironaka's flattening theorem cf. \cite[Lemma 7.3]{Vieh83}, we can find a morphism $f' : X' \to Z'$ between two compact K\"ahler manifolds and satisfies the following commutative diagram
\begin{center}
\begin{tikzcd}
X' \arrow[swap]{d}{\pi} \arrow{r}{f'} & Z' \arrow{d}{\mu}\\
X \arrow{r}{f} &  Z
\end{tikzcd}
\end{center}
such that the morphisms $\pi$ and $\mu$ are bimeromorphic, and moreover, each hypersurface $W\subset X'$ such that $\codim_{Y'} f' (W) \geq 2$ is $\pi$-contractible,
i.e., $\codim_X \pi (W) \geq 2$.

We denote by $\hat B$ the strict transform of $B$ by $\pi$, and write $K_{X'}+\hat B=\pi^*(K_X+B)+\sum a_i E_i$. We set 
$B':=\hat B+\sum_{a_i<0}(-a_i)E_i$. Then $(X', B')$ is klt.
Let us choose $m$ large enough so that $\cF_m := f' _\star (m (K_{X'/Z'} +B'))$ is non-zero. Then, $\cF_m$ is a torsion free sheaf of rank one on $Z'$ and its reflexive hull $\cF_m ^{\star\star}$ is a line bundle that we can equip with the $m$-th Narashimhan-Simha metric $h_m$. Thanks to Remark \ref{indepofm}, 
the following $\Q$-line bundle and the metric 
$$L  :=  \frac 1 m  f' _\star (m (K_{X'/Z'} +B'))^{\star\star} ,\qquad h := h_{m} ^{\frac{1}{m}}$$
are independent of the choice of $m$.
Let $\phi$ be the weight of $h$.
\medskip

\noindent We have the following statement, which connects our current setting with Definition/Proposition
\ref{KElog}.
\begin{prop}
\label{integrali} 
In the above setting, the following holds. 
\begin{enumerate}
\item  \label{bp} $i\Theta_{\phi} (L) \geq 0$ on $Z'$.
\item $K_{Z'} + L$ is a big $\mathbb{Q}$-line bundle and for any $m\in \mathbb{N}$ sufficiently divisible, the algebra 
$$\bigoplus_{p\geq 0} H^0 (Z', pm (K_{Z'}+L))$$ is finitely generated.
\item For every $p\in \N$ sufficiently divisible and every $s\in H^0 (Z', p ( K_{Z'} + L))$, we have 
$$\int_{Z'} |s|^{\frac{2}{p}} e^{-\phi}< +\infty .$$
\end{enumerate}
\end{prop}

\begin{proof}
The first item is a direct consequence of \cite{BP}. 

\medskip

\noindent
For the second term,  let $m\in \N$ sufficiently divisible such that for every $p\in\N$, $H^0 (X', p m (K_{X'} +B'))$ is generated by $\otimes^p H^0 (X',  m (K_{X'} +B'))$.
By the construction of $L$, there exist two effective divisors $E_{+}$ and $E_{-}$ on $X'$ such that 
\begin{equation}
\label{equalitydirect}
m (K_{X'} +B') + E_{-} = m \cdot (f')^* ( K_{Z'} +L) + E_{+} ,
\end{equation}
and for every $\tau\in H^0 (X' , m (K_{X'} +B'))$, $\tau$ vanishes over $E_{+}$.  
As a consequence, for every $s\in H^0 (X', pm (K_{X'} +B'))$, $s$ vanishes over $p [E_{+}]$.
Note that $(f')_* (E_{-})$ is supported in the non locally free locus of $f' _\star (m (K_{X'/Z'} +B'))$. Then $E_{-}$ is $\pi$-contractible. Together with the above argument, for every $p\in\N$, we have the natural isomorphisms 
$$H^0 (Z', pm (K_{Z'} +L) ) = H^0 (X',  pm (K_{X'} +B') + p E_{-}) =H^0 (X, pm (K_X +B)) .$$
As a consequence,  $K_{Z'} +L$ is big and the algebra 
$$\bigoplus_{p\geq 0} H^0 (Z', pm (K_{Z'}+L))$$ is finitely generated.

\medskip

\noindent
For the third term,
let $s_{E_{+}}$ be the canonical section of $E_{+}$ and let $h_{B'}$ (resp. $h_{E_{-}}$) be the canonical singular metric on $B'$ (resp. $E_{-}$). As $p$ is sufficient divisible, we can assume that $p_1 :=\frac{p}{m}\in\mathbb{N}$.
Thanks to \eqref{equalitydirect}, we have 
$$(f')^* (s) \otimes s_{E_{+}} ^{\otimes p_1}\in H^0 (X', p_1 ( m K_{X'} + m B' + E_{-})) .$$
By the definition of the Narashimhan-Simha metric, we have
$$\int_{Z'} |s|^{\frac{2}{p}} e^{-\phi} =\int_{X'} |(f')^* (s) \otimes s_{E_{+}} ^{\otimes p_1}|^\frac{2}{p} _{h_{B'} , h_{E_{-}}} .$$
Note that $E_{-}$ is $\pi$-contractible, $(f')^* (s) \otimes s_{E_{+}} ^{\otimes p_1}$ vanishes along $E_{-}$ of order at least $p_1 E_{-}$.
Together with the fact that $B'$ is klt, we have 
$$\int_{X'} |(f')^* (s) \otimes s_{E_{+}} ^{\otimes p_1}|^\frac{2}{p} _{h_{B'}, h_{E_{-}}} < +\infty .$$
The proposition is proved.
\end{proof}

\begin{rema}
Using the above argument, we know that $e^{-\phi}$ is in $L^1 _{\rm loc} (Z' \setminus f'  (E_{-}))$.
\end{rema}

Together with Proposition~\ref{integrali}, we get

\begin{coro}
\label{cor:ns}
With the above notation,  $Z'$ admits a natural Kähler-Einstein metric $\omke$ in the sense of Definition~\ref{KElog}. This metric satisfies 
\begin{equation}
\label{omns}
\Ric \omke = -\omke+  \Theta_{\phi}(L) \quad \mbox{on} \, \, Z'. 
\end{equation}
We call $\omke$ the K\"ahler-Einstein metric assoicated to the Iitaka fibration of $(X, B)$.
\end{coro}
\bigskip

\subsection{Relation with the canonical metrics.}\label{relatsubsect}

Let $X$ be a compact K\"ahler manifold and let $B$ be a $\Q$-divisor with snc support such that $(X,B)$ is klt. We suppose that $\kappa (K_X +B) \geq 1$. Thanks to \cite{BCHM,Fuj13}, 
the canonical model $Z$ of $(X, B)$ is normal. After blowing up the indeterminacy locus of the Iitaka fibration, we can suppose that
the Iitaka fibration of $K_X+B$ induces a morphism $f:X\to Z$
and there is  an ample $\Q$-line bundle $A$ on $Z$ such that 
$$K_X +B = f^* A +E_X$$
is a Zariski decomposition for some effective $\mathbb{Q}$-divisor $E_X$ with normal crossing support on $X$. In that context, the analogue of Kähler-Einstein metrics for the pair $(X,B ,E_X)$, sometimes called canonical metrics, are objects that are singular metrics $\omcan$ on $Z$ satisfying a "canonical" Monge-Ampère equation. They were first introduced by Song-Tian when $B=E_X=0$ \cite{ST12} and later generalized by Eyssidieux-Guedj-Zeriahi \cite[Definition 2.2, 2.7]{EGZ16}.

\medskip

Let us recall the definition of the canonical metric in this setting.
One first picks a smooth hermitian $h_A=e^{-\phi_A}$ on $A$ with positive curvature $\chi:=\ddc \phi_A$. Then, one introduces a measure $\mu_{h_A, h_E}$ on $X$ by setting $$\mu_{h_A, h_E}:=\frac{(\sigma \wedge \bar \sigma)^{\frac 1 N}e^{-\phi_B}}{|\sigma|_{f^*h_A, h_E}^{2/N}}$$ where $\sigma$ is a local trivialization of $N(K_X+B)$ for $N$ divisible enough, $\phi_B$ is the canonical singular weight on $B$ and $h_E$ is the canonical singular metric on $E$. Finally, one defines $\omcan:=\chi+\ddc \varphi_{\rm can}$ as the unique positive current on $Z$ with bounded potentials such that 
\begin{equation}
\label{canmet}
(\chi+\ddc \varphi_{\rm can})^{\dim Z} = e^{\varphi_{\rm can}}f_*\mu_{h_A, h_E} .
\end{equation}
Note that the singularity of $h_E$ gives rise the zero locus of $\mu_{h_A, h_E}$. One can check that the measure $f_*\mu_{h_A, h_E}$ has $L^{1+\ep}$ density with respect to a smooth volume form, cf. \cite[Lemma 2.1]{EGZ16}. Moreover, the canonical metric $\omcan$ is independent of 
the choice the hermitian metric $h_A$, cf. \cite[Lemma 2.4]{EGZ16}.

\medskip

By applying the construction in subsection \ref{ssec:intkod} to $f: X\rightarrow Z$, 
we can find a morphism $f' : X' \to Z'$ between two compact K\"ahler manifolds and satisfies the following commutative diagram
\begin{center}
\begin{tikzcd}
X' \arrow[swap]{d}{\pi} \arrow{r}{f'} & Z' \arrow{d}{\mu}\\
X \arrow{r}{f} &  Z
\end{tikzcd}
\end{center}
such that the morphisms $\pi$ and $\mu$ are bimeromorphic, and moreover, each hypersurface $W\subset X'$ such that $\codim_{Y'} f' (W) \geq 2$ is $\pi$-contractible,
i.e., $\codim_X \pi (W) \geq 2$. Following the notations in Proposition \ref{integrali},
for $m$ large enough, we can equip $L :=\frac 1 m  f' _\star (m (K_{X'/Z'} +B'))^{\star\star}$ with the  Narashimhan-Simha type metric $e^{-\phi}$.
Thanks to Corollary~\ref{cor:ns}, we can find the "NS-type" Kähler-Einstein metric $\omke$ on $Z'$ which satisfies
$$\Ric \omke = -\omke+  \frac{i}{2\pi}\Theta_{\phi}(L) \quad \mbox{on} \, \, Z'. $$

\medskip

We now establish a relation between the Kähler-Einstein metric $\omke$ and the canonical metric $\omcan$. 

\begin{prop}
\label{sameke}
With the notation above, let $\omke$ be the Kähler-Einstein metric on $Z'$ solution of \eqref{omns} and let $\omcan$ be the canonical metric on $Z$ solution of \eqref{canmet}. 
Then, one has $\mu_*\omke = \omcan$. 
\end{prop}

\begin{proof}
Recall that by \eqref{equalitydirect}, we have
\begin{equation}
\label{addar}
m (K_{X'} +B') + E_{-} = (f')^* (m( K_{Z'} +L)) + E_{+} ,
\end{equation}
and by construction, we have
\begin{equation}
\label{addar2}
m(K_{X'}+B'-E_{X'})= (f'\circ\mu)^* mA ,
\end{equation}
for some $\Q$-effective divisor $E_{X'}$ such that $\pi_\star (E_{X'})= E_X$.

\medskip

We first establish the relation between $\mu^*A$ and $K_{Z'}+L$.
Remember that $f'_*\mathcal O_{X'}(kE_+)\simeq \mathcal O_{Z'}$ for any integer $k$ such that $kE_+$ has integral coefficients. We deduce that $m(K_{Z'}+L) = \mu^*(m A) \otimes f'_*(mE_{X'}+E_-)$. In particular, we get that $f'_*(mE_{X'}+E_-)$ is a locally trivial sheaf of rank one, hence associated to a divisor $m E_{Z'}$; it is clearly effective and $\mu$-exceptional, as $\Supp (f')_* (m E_{X'}) \subset \mu^{-1}(Z_{\rm sing})$ and $\codim_{Z'}f'(E_-) \ge 2$. 
Then we have the Zariski decomposition
\begin{equation}
\label{zar2}
K_{Z'}+L \equiv_\Q \mu^*A+E_{Z'}
\end{equation}
By construction, we have $f'^*(mE_{Z'})+E_+=mE_{X'}+E_-$.

\medskip

Now, take $U\subset Z'$ a small coordinate open subset, and let $e_{\mu^* mA}\in H^0(U,\mu^* mA)$ and $e_{mE_{Z'}}\in H^0(U, mE_{Z'})$ be trivializations of $\mu^* mA$ and $mE_{Z'}$ respectively. Let $d\underline z$ be a trivialisation of $K_{Z'}$ over $U$. 
They induce a trivialization $e\in H^0(U, mL)$ of $m L$ such that 
\begin{equation}\label{trivequat}
d\underline z^{\otimes m}\otimes e = e_{\mu^* mA}\otimes e_{mE_{Z'}} .
\end{equation} 
Set $e^{-\varphi_{\mu^*A}}:=|e_{\mu^* mA}|^\frac{2}{m} _{\mu^* m h_A}$, $e^{-\vp_{E_{Z'}}}:= |e_{mE_{Z'}}|^{2/m}_{e^{-\phi_{mE_Z'}}}$ and $e^{-\varphi}:=|e|^{2/m}_{e^{-\phi}}$. 
Let 
$$\sigma:=(f')^*e_{\mu^* mA}\in H^0(f'^{-1}(U),m( K_{X'}+B'-E_{X'})) .$$ 
Thanks to \eqref{addar} and \eqref{trivequat},  we have
$$\tau:=\sigma\otimes (f')^*( e_{m Z'}) \otimes  s_{E_+} \in H^0(f'^{-1}(U), m(K_{X'}+B')+E_-) ,$$
and 
\begin{equation}\label{trivequatadd}
e^{-\vp (z)}=\int_{X'_z} |\tau|^{2/m}e^{-\phi_{B'}-\frac 1 m \phi_{E_-} }
= e^{-\vp_{E_{Z'}}}\cdot \int_{X'_z} |\sigma|^{2/m}e^{-\phi_{B'}+\phi_{E_{X'}}} . 
\end{equation}

\medskip

The canonical measure $\nu$ on $Z'$ has density with respect to the Lebesgue measure $d\lambda= |d\underline z|^2$ given by the formula 
$$\frac{d\nu}{d\lambda}(z) = \int_{X'_z} \frac{(\sigma \wedge \bar \sigma)^{1/m}e^{- \phi_{B'}+\phi_{E_{X'}} }}{|\sigma|^\frac{2}{m}_{f'^*\mu^* h_{A}}} = e^{\vp_{\mu^*A}} \int_{X'_z} |\sigma|^{2/m}e^{ - \phi_{B'}+\phi_{E_{X'}}}$$
for $z\in Z'$ generic. Together with \eqref{trivequatadd}, we get
$$\frac{d\nu}{d\lambda} = e^{\vp_{\mu^*A}- \vp + \vp_{E_{Z'}}}$$
Therefore, $\om:=\mu^*\omcan+[mE_{Z'}]$ satisfies
$$\Ric \om  = - \mu^*\omcan + \ddc \vp - [E_{Z'}] = -\om + \frac{i}{2\pi}\Theta_{\phi}(L)$$
As $\om\in c_1(K_{Z'}+L)$ has minimal singularities by the Zariski decomposition \eqref{zar2} and satisfies the same Monge-Ampère equation as $\omke$, one deduces that $\om=\omke$, i.e., 
$$\omke = \mu^*\omcan +[E_{Z'}] .$$ 
As $E_{Z'}$ is $\mu$-exceptional, the Proposition is proved.  
\end{proof}

\begin{rema}
The proof of Proposition~\ref{sameke} above shows the more precise identity $\omke= \mu^*\omcan +E_{Z'}$ for some explicit divisor $E_{Z'}$ on $Z'$. 
\end{rema}

\subsection{Relative Kähler-Einstein and canonical metrics, main Theorem } 
\label{ssec:relative}

To finish this section, we now discuss the positivity of the relative Kähler-Einstein or the canonical metrics when the fiber is of intermediate Kodaira dimension.

\begin{theo}
\label{corb}
Let $p: X\to Y$ be a projective fibration between two K\"ahler manifolds of relative dimension $n$ and let $B$ be an effective klt $\Q$-divisor on $X$. We assume that for a generic fiber $X_y$, the log Kodaira dimension satisfies $\kappa (K_{X_y}+B_y) >0$.  
Let $f: X\dashrightarrow Z$ be the relative Iitaka fibration of $K_{X/Y}+B$, and let $f':X'\rightarrow Z'$ a birational model of $f$ such that $X'$ and $Z'$ are smooth.
\begin{center}
\begin{tikzcd}
X'\arrow{d} \arrow{rr}{f'} & & Z' \arrow{d}\\
X\arrow[dashed]{rr}{f}\arrow[swap]{rd}{p} & & Z \arrow{dl}\\
& Y &
\end{tikzcd}
\end{center}
For $y$ generic, let $\omega_{{\rm can}, y}$ be the canonical metric on $Z'_y$ of the pair $(X'_y,B'_y)$; it induces a current $\omcan°$ over the smooth locus of $Z'\to Y$. 

\noindent
Assuming that Conjecture~\ref{conj} holds, then the current $\omcan°$ is positive and extends canonically to a closed positive current on $Z'$.

\end{theo}

The proof of Theorem~\ref{corb} consists mostly in putting together all the constructions explained above. By using \cite{BCHM, Fuj13}, the canonical ring 
$$R(X_y ,B_y)=\bigoplus_{m\ge 0 } H^0(X_y ,\lfloor m(K_{X_y} +B_y) \rfloor)$$ 
is finitely generated.
Together with the fact that $h^0 (X_y , m (K_{X_y} +B|_{X_y}))$ is constant with respect to $y \in Y°$ for some $m$ large enough,  
one can construct the relative Iitaka fibration $f:X\dashrightarrow Z:=\mathrm{Proj}(p_*(m(K_{X/Y}+B))$. Thanks to the Subsection~\ref{relatsubsect}, 
we can find a desingularization $f' : X'\to Z'$ fitting the commutative diagram
\begin{center}
\begin{tikzcd}
X'\arrow{d}[swap]{\pi} \arrow{rr}{f'} & & Z' \arrow{d}{\mu}   \arrow[ddl, bend left = 80, "q'"]\\
X\arrow[dashed]{rr}{f}\arrow[swap]{rd}{p} & & Z \arrow{dl}{q}\\
& Y &
\end{tikzcd}
\end{center}
such that we have a $f'$-Zariski decomposition over $Y_0$
$$ K_{X'} +B' \equiv_\Q  (f'\circ \mu)^* A+E \qquad\text{on} \quad  (f')^{-1} (Y_0) ,$$
where $A$ is $q$-ample. 

Let $q':=q \circ \mu:Z'\to Y$ be the projection to the base. Each fiber $Z'_y$ for $y\in Y°$ can be endowed with a canonical metric $\om_{{\rm can}, y} \in c_1(\mu^*A_y)$ and a Kähler-Einstein metric $\om_{{\rm KE},y}\in c_1(K_{Z'_y}+L_y)$, where $L_y :=L |_{Z'_y}$ is the restriction of the $\Q$-line bundle $L:=\frac 1m f'_*(m(K_{X'/Z'}+B'))^{**}$ to $Z'_y$, endowed with the corresponding restriction of the Narasimhan-Simha metric on $L$.

 In particular, these fiberwise metrics induce singular hermitian metrics $e^{-\phi_{\rm can}}$ on $\mu^*A|_{q^{-1}(Y°)}$ and $e^{-\phi_{\rm KE}}$ on $L|_{q'^{-1}(Y°)}$ respectively. As seen in the Subsection~\ref{relatsubsect}, there exists a $\mu$-exceptional effective divisor $E_{Z'}$ on $Z'$ such that $\phi_{\rm KE}=\phi_{\rm can}+ [E_{Z'}]$.

 Assuming Conjecture~\ref{conj}, $e^{-\phi_{\rm KE}}$ is a positively curved metric on $(K_{Z'/Y}+L)|_{q'^{-1}(Y°)}$ that extends canonically to a positively curved metric on $K_{Z'/Y}+L$ on the whole $Z'$. As $\phi_{\rm can}$ comes from $Z$ and $E_{Z'}$ is $\mu$-exceptional, if follows that $e^{-\phi_{\rm can}}$ is a positively curved metric on $\mu^*A|_{q'^{-1}(Y°)}$ that extends canonically to $Z'$. This proves Theorem~\ref{corb}. 


\bibliographystyle{smfalpha}
\bibliography{biblio}

\newcommand{\etalchar}[1]{$^{#1}$}
\providecommand{\bysame}{\leavevmode ---\ }
\providecommand{\og}{``}
\providecommand{\fg}{''}
\providecommand{\smfandname}{\&}
\providecommand{\smfedsname}{\'eds.}
\providecommand{\smfedname}{\'ed.}
\providecommand{\smfmastersthesisname}{M\'emoire}
\providecommand{\smfphdthesisname}{Th\`ese}
\begin{thebibliography}{BCHM10}

\bibitem[Aub78]{Aubin}
{\scshape T.~Aubin} -- {\og \'{E}quations du type {M}onge-{A}mp\`ere sur les
  vari\'et\'es k\"ahl\'eriennes compactes\fg}, \emph{Bull. Sci. Math. (2)}
  \textbf{102} (1978), no.~1, p.~63--95.

\bibitem[BBE{\etalchar{+}}11]{BBEGZ}
{\scshape R.~J. Berman, S.~Boucksom, P.~Eyssidieux, V.~Guedj {\normalfont
  \smfandname} A.~Zeriahi} -- {\og {K{\"a}hler-Einstein metrics and the
  K{\"a}hler-Ricci flow on log-Fano varieties}\fg}, Preprint
  \href{http://arxiv.org/abs/1111.7158}{arXiv:1111.7158}, to appear in J. Reine
  Angew. Math., 2011.

\bibitem[BCHM10]{BCHM}
{\scshape C.~Birkar, P.~Cascini, C.~Hacon {\normalfont \smfandname}
  J.~McKernan} -- {\og {Existence of minimal models for varieties of log
  general type}\fg}, \emph{J. Amer. Math. Soc.} \textbf{23} (2010),
  p.~405--468.

\bibitem[BDIP96]{BDIP}
{\scshape J.~Bertin, J.-P. Demailly, L.~Illusie {\normalfont \smfandname}
  C.~Peters} -- \emph{Introduction \`a la th\'{e}orie de {H}odge}, Panoramas et
  Synth\`eses [Panoramas and Syntheses], vol.~3, Soci\'{e}t\'{e}
  Math\'{e}matique de France, Paris, 1996.

\bibitem[BEGZ10]{BEGZ}
{\scshape S.~Boucksom, P.~Eyssidieux, V.~Guedj {\normalfont \smfandname}
  A.~Zeriahi} -- {\og {Monge-Amp{\`e}re equations in big cohomology
  classes.}\fg}, \emph{Acta Math.} \textbf{205} (2010), no.~2, p.~199--262.

\bibitem[Ber15]{BoB15}
{\scshape B.~Berndtsson} -- {\og The openness conjecture and complex
  {B}runn-{M}inkowski inequalities\fg}, in \emph{Complex geometry and
  dynamics}, Abel Symp., vol.~10, Springer, Cham, 2015, p.~29--44.

\bibitem[BG14]{BG}
{\scshape R.~J. Berman {\normalfont \smfandname} H.~Guenancia} -- {\og
  {K{\"a}hler-Einstein metrics on stable varieties and log canonical
  pairs}\fg}, \emph{Geometric and Function Analysis} \textbf{24} (2014), no.~6,
  p.~1683--1730.

\bibitem[BP08]{BP}
{\scshape B.~{Berndtsson} {\normalfont \smfandname} M.~{P\u aun}} -- {\og
  {Bergman kernels and the pseudoeffectivity of relative canonical
  bundles.}\fg}, \emph{{Duke Math. J.}} \textbf{145} (2008), no.~2, p.~341--378
  (English).

\bibitem[BP12]{BP2}
{\scshape B.~Berndtsson {\normalfont \smfandname} M.~P{\u{a}}un} -- {\og
  Quantitative extensions of pluricanonical forms and closed positive
  currents\fg}, \emph{Nagoya Math. J.} \textbf{205} (2012), p.~25--65.

\bibitem[Bre13]{Brendle}
{\scshape S.~Brendle} -- {\og {Ricci flat K{\"a}hler metrics with edge
  singularities }\fg}, \emph{International Mathematics Research Notices}
  \textbf{24} (2013), p.~5727--5766.

\bibitem[BS17]{BS}
{\scshape M.~Braun {\normalfont \smfandname} G.~Schumacher} -- {\og {Kähler
  forms for families of Calabi-Yau manifolds}\fg}, Preprint
  \href{https://arxiv.org/abs/1702.07886}{arXiv:1702.07886}, 2017.

\bibitem[CGP13]{CGP}
{\scshape F.~Campana, H.~Guenancia {\normalfont \smfandname} M.~P\u{a}un} --
  {\og {Metrics with cone singularities along normal crossing divisors and
  holomorphic tensor fields}\fg}, \emph{Ann. Scient. Éc. Norm. Sup.}
  \textbf{46} (2013), p.~879--916.

\bibitem[CGZ13]{CGZ}
{\scshape D.~Coman, V.~Guedj {\normalfont \smfandname} A.~Zeriahi} -- {\og
  Extension of plurisubharmonic functions with growth control\fg}, \emph{J.
  Reine Angew. Math.} \textbf{676} (2013), p.~33--49.

\bibitem[Cho15]{Choi}
{\scshape Y.-J. Choi} -- {\og {Semi-positivity of fiberwise Ricci-flat metrics
  on Calabi-Yau fibrations}\fg}, Preprint
  \href{http://arxiv.org/abs/1508.00323}{arXiv:1508.00323}, 2015.

\bibitem[Dem92]{D2}
{\scshape J.-P. Demailly} -- {\og Regularization of closed positive currents
  and intersection theory\fg}, \emph{J. Algebraic Geom.} \textbf{1} (1992),
  no.~3, p.~361--409.

\bibitem[DLM12]{DLM12}
{\scshape X.~Dai, K.~Liu {\normalfont \smfandname} X.~Ma} -- {\og A remark on
  weighted {B}ergman kernels on orbifolds\fg}, \emph{Math. Res. Lett.}
  \textbf{19} (2012), no.~1, p.~143--148.

\bibitem[DPS01]{DPS01}
{\scshape J.-P. Demailly, T.~Peternell {\normalfont \smfandname} M.~Schneider}
  -- {\og Pseudo-effective line bundles on compact {K}\"ahler manifolds\fg},
  \emph{Internat. J. Math.} \textbf{12} (2001), no.~6, p.~689--741.

\bibitem[EGZ09]{EGZ}
{\scshape P.~Eyssidieux, V.~Guedj {\normalfont \smfandname} A.~Zeriahi} -- {\og
  {Singular K{\"a}hler-Einstein metrics}\fg}, \emph{{J. Amer. Math. Soc.}}
  \textbf{22} (2009), p.~607--639.

\bibitem[EGZ16]{EGZ16}
{\scshape P.~Eyssidieux, V.~Guedj {\normalfont \smfandname} A.~Zeriahi} -- {\og
  {Convergence of weak Kähler-Ricci flows on minimal models of positive
  Kodaira dimension}\fg}, Preprint
  \href{http://arxiv.org/abs/1604.07001}{arXiv:1604.07001}, 2016.

\bibitem[Fuj15]{Fuj13}
{\scshape O.~Fujino} -- {\og Some remarks on the minimal model program for log
  canonical pairs\fg}, \emph{J. Math. Sci. Univ. Tokyo} \textbf{22} (2015),
  no.~1, p.~149--192.

\bibitem[GLZ18]{GLZ}
{\scshape V.~Guedj, C.~H. Lu {\normalfont \smfandname} A.~Zeriahi} -- {\og
  {Stability of solutions to complex Monge-Amp{è}re flows}\fg}, Preprint
  \href{http://arxiv.org/abs/1810.02123}{arXiv:1810.02123}, 2018.

\bibitem[Gue13]{G2}
{\scshape H.~Guenancia} -- {\og K\"ahler-{E}instein metrics with cone
  singularities on klt pairs\fg}, \emph{Internat. J. Math.} \textbf{24} (2013),
  no.~5, p.~1350035, 19.

\bibitem[Gue14]{G12}
{\scshape H.~Guenancia} -- {\og {K{\"a}hler-Einstein metrics with mixed
  Poincar{\'e} and cone singularities along a normal crossing divisor}\fg},
  \emph{Ann. Inst. Fourier} \textbf{64} (2014), no.~6, p.~1291--1330.

\bibitem[Gue16]{Gue16}
\bysame , {\og {Families of conic K{\"a}hler-Einstein metrics}\fg}, Preprint
  \href{http://arxiv.org/abs/1605.04348}{arXiv:1605.04348}, to appear in Math.
  Annalen, 2016.

\bibitem[GZ15]{GuanZhou15}
{\scshape Q.~Guan {\normalfont \smfandname} X.~Zhou} -- {\og A proof of
  {D}emailly's strong openness conjecture\fg}, \emph{Ann. of Math. (2)}
  \textbf{182} (2015), no.~2, p.~605--616.

\bibitem[HS17]{HS}
{\scshape H.-J. Hein {\normalfont \smfandname} S.~Sun} -- {\og Calabi-{Y}au
  manifolds with isolated conical singularities\fg}, \emph{Publ. Math. Inst.
  Hautes \'Etudes Sci.} \textbf{126} (2017), p.~73--130.

\bibitem[JMR16]{JMR}
{\scshape T.~Jeffres, R.~Mazzeo {\normalfont \smfandname} Y.~A. Rubinstein} --
  {\og K\"ahler-{E}instein metrics with edge singularities\fg}, \emph{Ann. of
  Math. (2)} \textbf{183} (2016), no.~1, p.~95--176, with an Appendix by C. Li
  and Y. Rubinstein.

\bibitem[Kob84]{KobR}
{\scshape R.~Kobayashi} -- {\og {K{\"a}hler-Einstein metric on an open
  algebraic manifolds}\fg}, \emph{Osaka 1. Math.} \textbf{21} (1984),
  p.~399--418.

\bibitem[P{\u{a}}u17]{Paun12}
{\scshape M.~P{\u{a}}un} -- {\og Relative adjoint transcendental classes and
  {A}lbanese map of compact {K}\"{a}hler manifolds with nef {R}icci
  curvature\fg}, in \emph{Higher dimensional algebraic geometry---in honour of
  {P}rofessor {Y}ujiro {K}awamata's sixtieth birthday}, Adv. Stud. Pure Math.,
  vol.~74, Math. Soc. Japan, Tokyo, 2017, p.~335--356.

\bibitem[RT11]{RT11}
{\scshape J.~Ross {\normalfont \smfandname} R.~Thomas} -- {\og Weighted
  {B}ergman kernels on orbifolds\fg}, \emph{J. Differential Geom.} \textbf{88}
  (2011), no.~1, p.~87--107.

\bibitem[Sch08]{Schum08}
{\scshape G.~Schumacher} -- {\og Curvature of higher direct images and
  applications\fg}, Preprint
  \href{http://arxiv.org/abs/0808.3259}{arXiv:0808.3259}, 2008.

\bibitem[Sch12]{Schum}
{\scshape G.~Schumacher} -- {\og Positivity of relative canonical bundles and
  applications\fg}, \emph{Invent. Math.} \textbf{190} (2012), no.~1, p.~1--56.

\bibitem[ST12]{ST12}
{\scshape J.~Song {\normalfont \smfandname} G.~Tian} -- {\og Canonical measures
  and {K}\"ahler-{R}icci flow\fg}, \emph{J. Amer. Math. Soc.} \textbf{25}
  (2012), no.~2, p.~303--353.

\bibitem[Tsu10]{TsujiNag}
{\scshape H.~Tsuji} -- {\og Dynamical construction of {K}\"ahler-{E}instein
  metrics\fg}, \emph{Nagoya Math. J.} \textbf{199} (2010), p.~107--122.

\bibitem[Tsu11]{Tsuji10}
\bysame , {\og Canonical singular {H}ermitian metrics on relative canonical
  bundles\fg}, \emph{Amer. J. Math.} \textbf{133} (2011), no.~6, p.~1469--1501.

\bibitem[TY87]{Tia}
{\scshape G.~Tian {\normalfont \smfandname} S.-T. Yau} -- {\og {Existence of
  K{\"a}hler-Einstein metrics on complete K{\"a}hler manifolds and their
  applications to algebraic geometry}\fg}, \emph{Adv. Ser. Math. Phys. 1}
  \textbf{1} (1987), p.~574--628, Mathematical aspects of string theory (San
  Diego, Calif., 1986).

\bibitem[Vie83]{Vieh83}
{\scshape E.~Viehweg} -- {\og Weak positivity and the additivity of the
  {K}odaira dimension for certain fibre spaces\fg}, in \emph{Algebraic
  varieties and analytic varieties ({T}okyo, 1981)}, Adv. Stud. Pure Math.,
  vol.~1, North-Holland, Amsterdam, 1983, p.~329--353.

\bibitem[Yau78]{Yau78}
{\scshape S.-T. Yau} -- {\og {On the Ricci curvature of a compact K{\"a}hler
  manifold and the complex Monge-Amp{\`e}re equation. I.}\fg}, \emph{Commun.
  Pure Appl. Math.} \textbf{31} (1978), p.~339--411.

\end{thebibliography}

\end{document}